\documentclass{gtmon_a}
\pdfoutput=1

\usepackage{graphicx}

\proceedingstitle{Heegaard splittings of 3--manifolds (Haifa 2005)}
\conferencestart{10 July 2005}
\conferenceend{19 July 2005}
\conferencename{Heegaard splittings of 3--manifolds}
\conferencelocation{Haifa}

\editor{Cameron Gordon}
\givenname{Cameron}
\surname{Gordon}

\editor{Yoav Moriah}
\givenname{Yoav}
\surname{Moriah}

\title[The rank of the fundamental group of hyperbolic 3--manifolds]
{Geometry, Heegaard splittings and rank of the\\
fundamental group of hyperbolic 3--manifolds}

\author{Juan Souto}
\givenname{Juan}
\surname{Souto}
\address{Department of Mathematics \\
University of Chicago\\\newline
5734 S University Avenue\\
Chicago\\Illinois 60637\\USA}
\email{juan@math.uchicago.edu}
\urladdr{www.math.uchicago.edu/~juan}

\volumenumber{12}
\issuenumber{}
\publicationyear{2007}
\papernumber{15}
\startpage{351}
\endpage{399}

\doi{}
\MR{}
\Zbl{}

\arxivreference{}  


\received{24 October 2006}
\revised{2 May 2007}
\accepted{26 June 2007}
\published{}
\publishedonline{1 December 2007}
\proposed{}
\seconded{}
\corresponding{}
\version{}

\let\xysavmatrix\xymatrix
\def\xymatrix{\disablesubscriptcorrection\xysavmatrix}
\AtBeginDocument{\let\wbar\wbar\let\tilde\widetilde\let\hat\what}

\makeatletter
\def\cnewtheorem#1[#2]#3{\newtheorem{#1}{#3}[section]
\expandafter\let\csname c@#1\endcsname\c@sat}

\newtheorem{sat}{Theorem}[section]			
\cnewtheorem{lem}[sat]{Lemma}
\cnewtheorem{kor}[sat]{Corollary}				
\cnewtheorem{prop}[sat]{Proposition}
\cnewtheorem{bei}[sat]{Example}				

\newtheorem*{defi*}{Definition}				
\newtheorem*{bei*}{Example}
\newtheorem*{sat*}{Theorem}					
\newtheorem*{kor*}{Corollary}
\newtheorem*{rmk*}{Remark}

\renewcommand{\theequation}{\thesection.\arabic{equation}}
\let\ssection=\section
\renewcommand{\section}{\setcounter{equation}{0}\ssection}

\newtheorem*{HT}{Hyperbolization Theorem}
\newtheorem*{MRT}{Mostow's Rigidity Theorem}
\newtheorem*{CTh}{Covering Theorem}

\newtheorem{conj}{Conjecture}

\theoremstyle{remark}
\newtheorem*{bem}{Remark}

\newtheorem{prob}{Problem}

\makeatother  

\makeautorefname{sat}{Theorem}
\makeautorefname{kor}{Corollary}

\newcommand{\BC}{\mathbb C}			\newcommand{\BH}{\mathbb H}
\newcommand{\BR}{\mathbb R}			\newcommand{\BD}{\mathbb D}
\newcommand{\BN}{\mathbb N}			
\newcommand{\BS}{\mathbb S}			\newcommand{\BZ}{\mathbb Z}
\newcommand{\BF}{\mathbb F}

\newcommand{\CC}{\mathcal C}		\newcommand{\calD}{\mathcal D}
\newcommand{\CE}{\mathcal E}		
		\newcommand{\CH}{\mathcal H}
		
		\newcommand{\CL}{\mathcal L}
\newcommand{\CM}{\mathcal M}		\newcommand{\CN}{\mathcal N}
		\newcommand{\CP}{\mathcal P}
		
\newcommand{\CS}{\mathcal S}

\newcommand{\FI}{\mathfrak Is}

\newcommand{\bnd}{\partial}

\DeclareMathOperator{\PSL}{PSL}		
\DeclareMathOperator{\vol}{vol}		

\DeclareMathOperator{\Map}{Map}
\DeclareMathOperator{\inj}{inj}
\DeclareMathOperator{\diam}{diam}
\DeclareMathOperator{\rank}{rank}

\DeclareMathOperator{\Ric}{Ric}
\DeclareMathOperator{\rel}{rel}
\DeclareMathOperator{\arccosh}{arccosh}

\begin{document}

\begin{asciiabstract}
In this survey we discuss how geometric methods can be used to study
topological properties of 3-manifolds such as their Heegaard genus or
the rank of their fundamental group. On the other hand, we also
discuss briefly some results relating combinatorial descriptions and
geometric properties of hyperbolic 3-manifolds.
\end{asciiabstract}

\begin{abstract}
In this survey we discuss how geometric methods can be used to study
topological properties of 3--manifolds such as their Heegaard genus or
the rank of their fundamental group. On the other hand, we also
discuss briefly some results relating combinatorial descriptions and
geometric properties of hyperbolic 3--manifolds.
\end{abstract}

\maketitle

A closed, and say orientable, Riemannian 3--manifold $(M,\rho)$ is {\em
hyperbolic\/} if the metric $\rho$ has constant sectional curvature
$\kappa_\rho=-1$. Equivalently, there is a discrete and torsion free group
$\Gamma$ of isometries of hyperbolic 3--space $\BH^3$ such that the
manifolds $(M,\rho)$ and $\BH^3/\Gamma$ are isometric. It is well-known
that the fundamental group $\pi_1(M)$ of every closed 3--manifold which
admits a hyperbolic metric is a non-elementary Gromov hyperbolic group and
hence that it is is infinite and does not contain free abelian subgroups of
rank 2. A 3--manifold $M$ whose fundamental group does not have subgroups
isomorphic to $\BZ^2$ is said to be {\em atoroidal\/}. Another well-known
property of those 3--manifolds which admit a hyperbolic metric is that they
are {\em irreducible\/}, ie every embedded sphere bounds a ball.
Surprisingly, these conditions suffice to ensure that a closed 3--manifold
$M$ admits a hyperbolic metric.

\begin{HT}[Perelman]\hypertarget{HypTh}{A} 
closed orientable 3--manifold $M$ admits a hyperbolic metric if and only
if it is irreducible, atoroidal and has infinite fundamental group.
\end{HT}

Thurston proved the 
\hyperlink{HypTh}{Hyperbolization Theorem}
in many cases, for instance if
$M$ has positive first Betti-number (see Otal \cite{Otal96,Otal98}). The
\hyperlink{HypTh}{Hyperbolization Theorem}
is a particular case of Thurston's Geometrization
conjecture recently proved by Perelman \cite{Per1,Per2,Per3} (see also
Cao--Zhu \cite{Cao-Zhu}).

From our point of view, the \hyperlink{HypTh}{Hyperbolization Theorem} is only one half of the
coin, the other half being Mostow's Rigidity Theorem.

\begin{MRT}
\hypertarget{MRTh}{Any} 
two closed hyperbolic 3--manifolds which are homotopy equivalent are
isometric.
\end{MRT}

The goal of this note is to describe how the existence and uniqueness of
hyperbolic metrics can be used to obtain results about quantities which
have been classically studied in 3--dimensional topology. More precisely, we
are interested in the {\em Heegaard genus\/} $g(M)$ of a 3--manifold $M$ and
in the {\em rank\/} of its fundamental group. Recall that a {\em Heegaard
splitting\/} of a closed 3--manifold is a decomposition of the manifold into
two handlebodies with disjoint interior. The surface separating both
handlebodies is said to be the {\em Heegaard surface\/} and its genus is the
genus of the Heegaard splitting. Moise \cite{Moise} proved that every
topological 3--manifold admits a Heegaard splitting. The Heegaard genus
$g(M)$ of $M$ is the minimal genus of a Heegaard splitting of $M$. The {\em
rank\/} of the fundamental group of $M$ is the minimal number of elements
needed to generate it.

Unfortunately, the \hyperlink{HypTh}{Hyperbolization Theorem} only guarantees that a
hyperbolic metric exists, but it does not provide any further information
about this metric. This is why most results we discuss below are about
concrete families of 3--manifolds for which there is enough geometric
information available. But this is also why we discuss which geometric
information, such as the volume of the hyperbolic metric, can be read from
combinatorial information about for example Heegaard splittings.

The paper is organized as follows.
In \fullref{sec:topo} and \fullref{sec:hyp} we recall some well-known facts
about 3--manifolds, Heegaard splittings and hyperbolic geometry. In
particular we focus on the consequences of tameness and of Thurston's
covering theorem.

In \fullref{sec:minimal} we describe different constructions of minimal
surfaces in 3--manifolds, in particular the relation between Heegaard
splittings and minimal surfaces. The so obtained minimal surfaces are used
for example in \fullref{sec:genus} to give a  proof of the fact that the
mapping torus of a sufficiently high power of a pseudo-Anosov mapping class
of a closed surface of genus $g$ has Heegaard genus $2g+1$.

In \fullref{sec:carrier} we introduce carrier graphs and discuss some of
their most basic properties. Carrier graphs are the way to translate
questions about generating sets of the fundamental group of hyperbolic
3--manifolds into a geometric framework. They are used in
\fullref{sec:rank1} to prove that the fundamental group of the mapping
torus of a sufficiently high power of a pseudo-Anosov mapping class of a
closed surface of genus $g$ has rank $2g+1$.

In \fullref{sec:hossein} we determine the rank of the fundamental group and
the Heegaard genus of those 3--manifolds obtained by gluing two handlebodies
by a sufficiently large power of a {\em generic\/} pseudo-Anosov mapping
class. When reading this last sentence, it may have crossed through the
mind of the reader that this must somehow be the same situation as for the
mapping torus. And in fact, it almost is. However, there is a crucial
difference. It follows from the full strength of the geometrization
conjecture, not just the \hyperlink{HypTh}{Hyperbolization Theorem}, that the manifolds in
question are hyperbolic; however, a priori nothing is known about these
hyperbolic metrics. Instead of using the existence of the hyperbolic
metric, in Namazi--Souto
\cite{Namazi-Souto} we follow a different approach, described
below. We construct out of known hyperbolic manifolds a negatively curved
metric on the manifolds in question. In particular, we have full control of
the metric and our previous strategies can be applied. In principle, our
metric and the actual hyperbolic metric are unrelated.

So far, we have considered families of 3--manifolds for which we had a
certain degree of geometric control and we have theorems asserting that for
most members of these families something happens. In \fullref{sec:volume}
we shift our focus to a different situation: we describe a result due to
Brock and the author relating the volume of a hyperbolic 3--manifold with a
certain combinatorial distance of one of its Heegaard splittings.

Finally, in \fullref{sec:rank2} we describe the geometry of those thick
hyperbolic 3--manifolds whose fundamental group has rank 2 or 3; in this
setting we cannot even describe a conjectural model but the results hint
towards the existence of such a construction.

We conclude with a collection of questions and open problems in
\fullref{sec:questions}.

This note is intended to be a survey and hence most proofs are only
sketched, and this only in the simplest cases. However, we hope that these
sketches make the underlying principles apparent. It has to be said that
this survey is certainly everything but all-inclusive, and that the same
holds for the bibliography. We refer mostly to  papers read by the author,
and not even to all of them. Apart of the fact that many important
references are missing, the ones we give are not well distributed. For
example Yair Minsky and Dick Canary do not get the credit that they deserve
since their work is in the core of almost every result presented here. It
also has to be said, that this survey is probably superfluous for those
readers who have certain familiarity with (1) the work of Canary and
Minsky, (2) the papers 
\cite{Colding-deLellis,Lackenby-genus,Namazi,Rubinstein} 
by Tobias Colding and Camillo de Lellis, Marc Lackenby, Hossein
Namazi and Hyam Rubinstein, and (3) have had a couple of conversations,
about math, with Ian Agol, Michel Boileau and Jean-Pierre Otal. Having
collaborators such as Jeff Brock also helps.

%
%
%
%

\section{Some 3--dimensional topology}\label{sec:topo}
From now on we will only consider orientable 3--manifolds $M$ which are {\em
irreducible\/}, meaning that every embedded sphere bounds a ball. We will
also assume that our manifolds do not contain surfaces homeomorphic to the
real projective plane $\BR P^2$. This is not much of a restriction because
$\BR P^3$ is the only orientable, irreducible 3--manifold which contains a
copy of $\BR P^2$.

A surface $S$ in $M$ with non-positive Euler characteristic $\chi(S)\le 0$
is said to be {\em $\pi_1$--injective\/} if the induced homomorphism
$\pi_1(S)\to\pi_1(M)$ is injective. A surface is {\em incompressible\/} if it
is embedded and $\pi_1$--injective. An embedded surface which fails to be
incompressible is said to be {\em compressible\/}. The surface $S$ is said to
be {\em geometrically compressible\/} if it contains an essential
simple closed curve which bounds a disk $D$ in $M$ with $D\cap S=\bnd D$;
$D$ is said to be a {\em compressing disk\/}. Obviously, a geometrically
compressible surface is compressible. On the other hand there are
geometrically incompressible surfaces which fail to be incompressible.
However, the Loop theorem asserts that any such surface must be one-sided.
Summing up we have the following proposition.

\begin{prop}
A two-sided surface $S$ in $M$ is compressible if and only if it is
geometrically compressible.
\end{prop}

If $S$ is geometrically compressible and $D$ is a compressing disk then we
can obtain a new surface $S'$ as follows: we cut open $S$ along $\bnd D$
and glue to the obtained boundary curves two copies of $D$. We say that
$S'$ arises from $S$ by suturing along $D$. A surface is obtained from $S$
by {\em suturing along disks\/} if it is obtained by repeating this process
as often as necessary.

Given two embedded surfaces $S$ and $S'$ in $M$, we say that $S'$ arises
from $S$ by {\em collapsing along the normal bundle of $S'$\/} if there is a
regular neighborhood of $S'$ diffeomorphic to the total space of the normal
bundle $\pi \co N(S')\to S'$ containing $S$ and such that the restriction of
$\pi$ to $S$ is a covering of $S'$.

\begin{defi*}
Let $S$ and $S'$ be two embedded, possibly empty, surfaces in $M$. The
surface $S'$ arises from $S$ by {\em surgery\/} if it does by a combination
of isotopies, suturing along disks, discarding inessential spheres and
collapsing along the normal bundle of $S'$.
\end{defi*}

Observe that if the surface $S$ bounds a handlebody in $M$ then $\emptyset$
arises from $S$ by surgery. Similarly, the interior boundary of a
compression body arises from the exterior boundary by surgery. Recall that
a {\em compression body\/} $C$ is a compact orientable and irreducible
3--manifold which has a boundary component called the {\em exterior
boundary\/} $\bnd_e C$ such that the homomorphism $\pi_1(\bnd_eC)\to\pi_1(C)$
is surjective; the interior boundary is the union of all the remaining
boundary components. The genus of a compression body is the genus of its
exterior boundary.

A {\em Heegaard splitting\/} of the compact 3--manifold $M$ is a decomposition
$M=U\cup V$ into two compression bodies with disjoint interior and
separated by the corresponding exterior boundary, the so-called {\em
Heegaard surface} $\bnd_e U=\bnd_e V$. Heegaard splittings can be obtained
for example from Morse functions on $M$. Moise \cite{Moise} proved that
every topological 3--manifold admits a unique smooth structure; in
particular, every 3--manifold has a Heegaard splitting. The {\em Heegaard
genus\/} $g(M)$ of $M$ is the minimal genus of a Heegaard splitting of $M$.
It is a well-known fact that if $M$ is closed then $2g(M)+2$ is equal to
the minimal number of critical points of a Morse function on $M$.

As Morse functions can be perturbed to introduce new critical points new
Heegaard splittings can be obtained from other Heegaard splittings by
attaching new handles. A Heegaard splitting which is obtained from another
one by this process is said to be obtained by {\em stabilization\/}. It is
well-known that a Heegaard splitting $M=U\cup V$ arises by stabilization if
and only if there are two essential properly embedded disks $D_U\subset U$
and $D_V\subset V$ whose boundaries $\bnd D_U=\bnd D_V$ intersect in a
single point. A Heegaard splitting is {\em reducible\/} if there are two
essential properly embedded disks $D_U\subset U$ and $D_V\subset V$ with
$\bnd D_U=\bnd D_V$. Every reducible Heegaard splitting of an irreducible
3--manifold is stabilized.

Let $\Sigma\subset M$ be a (say) connected surface separating $M$ into two
components $N_1$ and $N_2$ and $f_1$ and $f_2$ be Morse functions on $N_1$
and $N_2$ whose values and derivatives of first and second order coincide
along $\Sigma$. Then the function $f \co M\to\BR$ given by $f(x)=f_i(x)$ if
$x\in N_i$ is a Morse function. Similarly, if a  3--manifold $M$ is
decomposed into codimension $0$ submanifolds $N_1,\dots,N_k$ with disjoint
interior then Heegaard splittings of $N_1,\dots,N_k$, fulfilling again some
normalization, can be merged to obtain a Heegaard splitting of $M$. See
Schultens \cite{Schultens} for a precise description of this process which is called
{\em amalgamation\/}. As in the case of stabilization, there is a criterium
to determine if a Heegaard splitting arises by amalgamation. One has namely
that this is the case for the Heegaard splitting $M=U\cup V$ if and only if
there are two essential properly embedded disks $D_U\subset U$ and
$D_V\subset V$ whose boundaries are disjoint. A Heegaard splitting $M=U\cup
V$ which is not reducible but such that there are two essential properly
embedded disks $D_U\subset U$ and $D_V\subset V$ with disjoint boundary is
said to be {\em weakly reducible\/}. A Heegaard splitting is {\em strongly
irreducible\/} if it is not reducible or weakly reducible. With this
terminology, we can summarize the above discussion as follows:
\begin{quote}
Every Heegaard splitting that is not strongly irreducible can be obtained
by amalgamation and stabilization from other splittings.
\end{quote}

Before stating a more precise version of this claim we need a last
definition.

\begin{defi*}
A {\em generalized Heegaard splitting\/} of $M$ is a pair of disjoint
embedded possibly disconnected surfaces $(\Sigma_I,\Sigma_H)$ such that the
following hold.
\begin{enumerate}
\item $\Sigma_I$ divides $M$ into two possibly disconnected manifolds $N_1$
and $N_2$.
\item The surfaces $\Sigma_1=\Sigma_H\cap N_1$ and $\Sigma_2=\Sigma_H\cap
N_2$ determine Heegaard splittings of $N_1$ and $N_2$ which can be
amalgamated to obtain a Heegaard splitting of $M$.
\end{enumerate}
A {\em generalized Heegaard splitting\/} $(\Sigma_I,\Sigma_H)$ of $M$ is {\em
strongly irreducible\/} if $\Sigma_I$ is incompressible and if the Heegaard
surface $\Sigma_H$ of $M\setminus\Sigma_I$ is strongly irreducible.
\end{defi*}

We have now the following crucial theorem,

\begin{sat}[Scharlemann--Thompson
\cite{Scharlemann-Thompson}]\label{all-generalized}
Every genus $g$ Heegaard splitting arises from first (1) amalgamating a
strongly irreducible generalized Heegaard splitting such that the involved
surfaces have at most genus $g$ and then (2) stabilizing the obtained
Heegaard splitting.
\end{sat}

This theorem is in some way a more precise version of the following result
of Casson and Gordon \cite{Casson-Gordon}.

\begin{sat}[Casson--Gordon]
If an irreducible 3--manifold $M$ admits a weakly reducible Heegaard
splitting then $M$ contains an incompressible surface.
\end{sat}

From our point of view, \fullref{all-generalized} asserts that most of the
time it suffices to study strongly irreducible splittings. However, the use
of \fullref{all-generalized} can be quite cumbersome because of the amount
of notation needed: for the sake of simplicity we will often just prove
claims for strongly irreducible splittings and then claim that the general
case follows using \fullref{all-generalized}.

As we just said, questions about Heegaard splittings can be often reduced
to questions about strongly irreducible splittings. And this is a lucky
fact since, while Heegaard splittings can be quite random, strongly
irreducible splittings show an astonishing degree of rigidity as shown for
example by the following lemma.

\begin{lem}{\rm (Scharlemann \cite[Lemma 2.2]{Scharlemann-intersections})} \qua
\label{lem:disks-in-strongly}
Suppose that an embedded surface $S$ determines a strongly irreducible
Heegaard splitting $M=U\cup V$ of a 3--manifold $M$ and that $D$ is an
embedded disk in $M$ transverse to $S$ and with $\bnd D\subset S$ then
$\bnd D$ also bounds a disk in either $U$ or $V$.
\end{lem}

From \fullref{lem:disks-in-strongly}, Scharlemann
\cite{Scharlemann-intersections} derived the following useful description
of the intersections of a strongly irreducible Heegaard surface with a
ball.

\begin{sat}[Scharlemann]
Let $M=U\cup V$ be a strongly irreducible Heegaard splitting with Heegaard
surface $S$. Let $B$ be a ball with $\bnd B$ transversal to $S$ and such
that the two surfaces $\bnd B\cap U$ and $\bnd B\cap V$ are incompressible
in $U$ and $V$ respectively. Then, $S\cap B$ is a connected planar surface
properly isotopic in $B$ to one of $U\cap\bnd B$ and $V\cap\bnd B$.
\end{sat}

In the same paper, Scharlemann also determined how a strongly irreducible
Heegaard splitting can intersect a solid torus.

Another instance in which the rigidity of strongly irreducible splittings
becomes apparent is the following lemma restricting which surfaces can
arise from a strongly irreducible surface by surgery.

\begin{lem}{(\rm Suoto \cite{Juan-book})} \qua \label{strongly}
Let $S$ and $S'$ be closed embedded surface in $M$. If $S$ is a strongly
irreducible Heegaard surface and $S'$ is obtained from $S$ by surgery and
has no parallel components then one of the following holds.
\begin{enumerate}
\item $S'$ is isotopic to $S$.
\item $S'$ is non-separating and there is a surface $\hat S$ obtained from
$S$ by surgery at a single disk and such that $\hat S$ is isotopic to the
boundary of a regular neighborhood of $S'$. In particular, $S'$ is
connected and $M\setminus S'$ is a compression body.
\item $S'$ is separating and $S$ is, up to isotopy, disjoint of $S'$.
Moreover, $S'$ is incompressible in the component $U$ of $M\setminus S'$
containing $S$, $S$ is a strongly irreducible Heegaard surface in $U$ and
$M\setminus U$ is a collection of compression bodies.
\end{enumerate}
\end{lem}

\fullref{strongly} is well-known to many experts. In spite of that we
include a proof.

\begin{proof}
Assume that $S'$ is not isotopic to $S$. We claim that in the process of
obtaining $S'$ from $S$, some surgery along disks must have been made.
Otherwise $S$ is, up to isotopy, contained in regular neighborhood
$\CN(S')$ of $S'$ such that the restriction of the projection $\CN(S')\to
S'$ is a covering. Since $S$ is connected, this implies that $S'$ is
one-sided and that $S$ bounds a regular neighborhood of $S'$. However, no
regular neighborhood of an one-sided surface in an orientable 3--manifold is
homeomorphic to a compression body. This contradicts the assumption that
$S$ is a Heegaard splitting.

We have proved that there is a surface $\hat S$ obtained from $S$ by
surgery along disks and discarding inessential spheres which is, up to
isotopy, contained in $\CN(S')$ and such that the restriction of the
projection $\CN(S')\to S'$ to $\hat S$ is a covering. Since $S$ determines
a strongly irreducible Heegaard splitting it follows directly from the
definition that all surgeries needed to obtain $\hat S$ from $S$ occur to
the same side. In particular, $\hat S$ divides $M$ into two components $U$
and $V$ such that $U$ is a compression body with exterior boundary $\hat S$
and such that $\hat S$ is incompressible in $V$. Moreover, $V$ is connected
and $S$ determines a strongly irreducible Heegaard splitting of $V$.

\begin{figure}[ht!]
\begin{center}
\includegraphics[scale=0.5,angle=270]{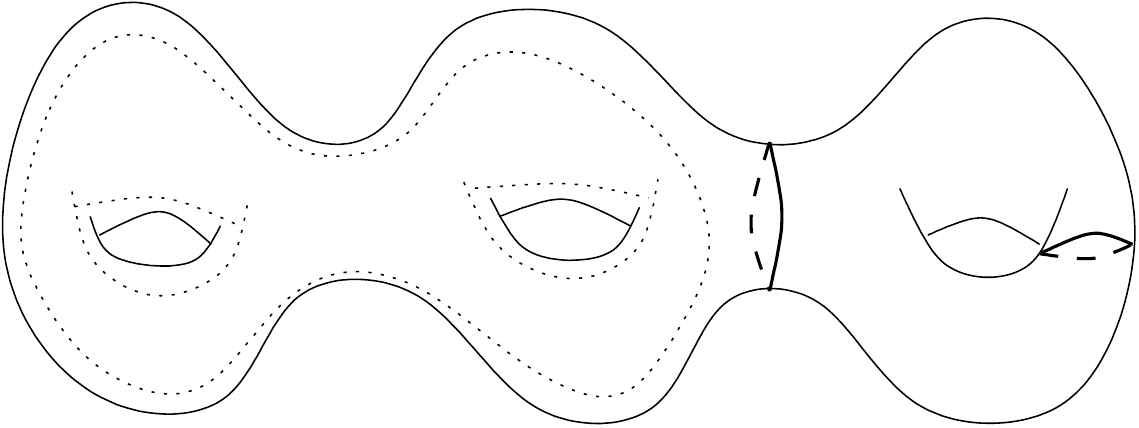}
\end{center}
\caption{The surface $S$ has genus $3$; the thick lines are the boundaries
of compressing disks; the surface $\hat S$ obtained from $S$ by surgery
along these disks is dotted.}
\label{fig:1}
\end{figure}

Assume that every component of $S'$ is two-sided. In particular, every
component of $\hat S$ is isotopic to a component of $S'$. If for every
component of $S'$ there is a single component of $\hat S$ isotopic to it,
then we are in case (3). Assume that this is not the case. Then there are
two components $\hat S_1$ and $\hat S_2$ which are isotopic to the same
component $S'_0$ and which bound a trivial interval bundle $W$ homeomorphic
to $S'_0\times[0,1]$ which does not contain any further component of $\hat
S$. In particular $W$ is either contained in $U$ or in $V$. Since the
exterior boundary of a compression body is connected we obtain that $W$
must be contained in $V$ and since $V$ is connected we have $W=V$. The
assumption that $S'$ does not have parallel components implies that $S'$ is
connected. It remains to prove that $S'$ arises from $S$ by surgery along a
single disk. The surface $S$ determines a Heegaard splitting of $V$ and
hence it is isotopic to the boundary of a regular neighborhood of $\hat
S\cup\Gamma$ where $\Gamma$ is a graph in $V$ whose endpoints are contained
in $\bnd V$. If $\Gamma$ is not a segment then it is easy to find two
compressible simple curves on $S$ which intersect only once (see \fullref{fig:2}).

\begin{figure}[ht!]
\begin{center}
\includegraphics[scale=0.5,angle=270]{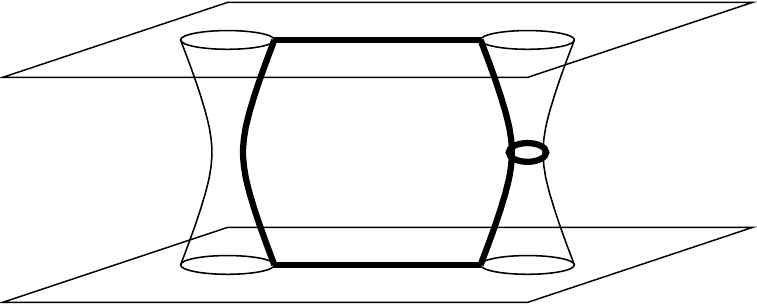}
\end{center}
\caption{Two disks intersecting once}
\label{fig:2}
\end{figure}

By \fullref{lem:disks-in-strongly} each one of these curves bounds a disk
in one of the components of $V\setminus S$ contradicting the assumption
that $S$ is strongly irreducible. This proves that $\hat S$ arises from $S$
by surgery along a single disk. We are done if every component of $\hat S$
is two-sided.

The argument in the case that $S'$ has at least a one-sided component is
similar.
\end{proof}

We conclude this section with some remarks about the curve complex $\CC(S)$
of a closed surface $S$ of genus $g\ge 2$ and its relation to Heegaard
splittings. The curve complex is the graph whose vertices are isotopy
classes of essential simple closed curves in $S$. The edges correspond to
pairs of isotopy classes that can be represented by disjoint curves.
Declaring every edge to have unit length we obtain a connected metric graph
on which the mapping class group $\Map(S)$ acts by isometries. If $S$ is
the a Heegaard surface then we define the {\em distance in the curve
complex\/} of the associated Heegaard splitting $M=U\cup V$ to be the minimal
distance between curves bounding essential disks in $U$ and in $V$. Using
this terminology, $M=U\cup V$ is reducible if the distance is $0$, it is
weakly reducible if the distance is $1$ and it is strongly irreducible if
the distance is at least $2$. The following result due to Hempel
\cite{Hem01} asserts that a manifold admitting a Heegaard splittings with
at least distance $3$ is irreducible and atoroidal.

\begin{sat}[Hempel]\label{Hempel}
If a 3--manifold admits a Heegaard splitting with at least distance 3 then
it is irreducible and atoroidal.
\end{sat}

For more on Heegaard splittings see Scharlemann \cite{Scharlemann-handbook}
and for generalized Heegaard
splittings Saito--Scharlemann--Schultens \cite{SSS}.

%
%
%
%

\section{Hyperbolic 3--manifolds and Kleinian groups}\label{sec:hyp}

A {\em hyperbolic structure\/} on a compact 3--manifold $M$ is the conjugacy
class of a discrete and faithful representation $\rho \co \pi_1(M)\rightarrow
\PSL_2(\BC)$, such that $N_\rho=\BH^3/\rho(\pi_1(M))$ is homeomorphic to
the interior of $M$ by a homeomorphism inducing $\rho$. From this point of
view, \hyperlink{MRTh}{Mostow's Rigidity Theorem} asserts that whenever $M$ is closed then
there is at most one hyperbolic structure.

It is never to early to remark that most results concerning hyperbolic
3--manifolds are still valid in the setting of manifolds of pinched negative
curvature (see for instance Canary \cite{Canary-tame} or Agol \cite{Agol}). 
In fact,
negatively curved metrics have a much greater degree of flexibility than
hyperbolic metrics and hence allow certain extremely useful constructions.
Further generalizations of hyperbolic metrics such as CAT(-1) metrics are
also ubiquitous; again because they are even more flexible than negatively
curved metrics. Recall that a geodesic metric space is CAT(-1) if, from the
point of view of comparison geometry it is at least as curved as hyperbolic
space (Bridson--Haefliger \cite{Bridson-Haffliger}).

The hyperbolic structure $N_\rho$ is {\em convex-cocompact\/} if there is a
convex $\rho(\pi_1(M))$--invariant subset $K\subset\BH^3$ with
$K/\rho(\pi_1(M))$ compact. Equivalently, the manifold $N_\rho$ contains a
compact convex submanifold $C$ such that $N_\rho\setminus C$ is
homeomorphic to $\bnd C\times\BR$.

If $S$ is a boundary component of $M$, then the {\em $S$--end\/} of $N_\rho$
is the end corresponding to $S$ under this homeomorphism. The end
corresponding to a component $S$ of $\bnd M$ is {\em convex-cocompact\/} if
$N_\rho$ contains a convex-submanifold $C$ such that $N_\rho\setminus C$ is
a neighborhood of the $S$--end of $N_\rho$ homeomorphic to $S\times\BR$.
Before going further we remind the reader of the following characterization
of the convex-cocompact structures.

\begin{lem}\label{cccq}
A hyperbolic structure $N_\rho$ is convex-cocompact if and only if for some
choice, and hence for all, of $p_{\BH^3}\in\BH^3$ the map
$$\pi_1(M)\to\BH^3, \qua \gamma\mapsto(\rho(\gamma))(p_{\BH^3})$$
is a quasi-isometric embedding. Here we endow $\pi_1(M)$ with the
left-invariant word-metric corresponding to some finite generating set.
\end{lem}

Recall that a map $\phi \co X_1\to X_2$ between two metric spaces is an
$(L,A)$--quasi-isometric embedding if
$$\frac 1Ld_{X_1}(x,y)-A\le d_{X_2}(\phi(x),\phi(y))\le Ld_{X_1}(x,y)+A$$
for all $x,y\in X_1$. An $(L,A)$--quasi-isometric embedding $\phi \co \BR\to X$
is said to be a quasi-geodesic.

Through out this note we are mostly interested in hyperbolic manifolds
without cusps. If there are cusps, ie if there are elements
$\gamma\in\pi_1(M)$ with $\rho(\gamma)$ parabolic, then there is an
analogous of the convex-cocompact ends: the geometrically finite ends.
Results about hyperbolic 3--manifolds without cusps can be often extended,
under suitable conditions and with lots of work, to allow general
hyperbolic 3--manifolds. We state this golden rule here.
\begin{quote}
Every result mentioned without cusps has an analogous result in the
presence of cusps.
\end{quote}

The geometry of convex-cocompact ends of $N_\rho$ is well-understood using
Ahlfors--Bers theory. Building on the work of Thurston, Canary
\cite{Canary-tame} described a different sort of end.

\begin{defi*}
An end $\CE$ of $N_\rho$ is {\em simply degenerate\/} if there is a sequence
of surfaces $(S_i)\subset N_\rho$ with the following properties.
\begin{itemize}
\item Every neighborhood of $\CE$ contains all but finitely many of the
surfaces $S_i$.
\item With respect to the induced path distance, the surface $S_i$ is
$CAT(-1)$ for all $i$.
\item If $C_1\subset N_\rho$ is a compact submanifold with $N_\rho\setminus
C_1$ homeomorphic to a trivial interval bundle, then $S_i$ is homotopic to
$\bnd C_1$ within $N_\rho\setminus C_1$.
\end{itemize}
\end{defi*}

The best understood example of manifolds with simply degenerate ends are
obtained as follows. A theorem of Thurston (see for example Otal \cite{Otal96})
asserts that whenever $S$ is a closed surface and $f\in\Map(S)$ is a
pseudo-Anosov mapping class on $S$, then the mapping torus
\begin{equation}\label{eq:map-torus}
M_f=(S\times[0,1])/((x,1)\simeq(f(x),0)
\end{equation}
admits a hyperbolic metric. The fundamental group of the fiber $\pi_1(S)$
induces an infinite cyclic cover $M_f'\to M_f$ homeomorphic to
$S\times\BR$. The surface $S\times\{0\}$ lifts to a surface, again denoted
by $S$, in $M_f'$ and there are many known ways to construct $CAT(-1)$
surfaces in $M_f'$ homotopic to $S$; for instance one can use simplicial
hyperbolic surfaces, minimal surfaces or pleated surfaces. For the sake of
concreteness, let $X$ be such a surface and $F$ be a generator of the deck
transformation group of the covering $M_f'\to M_f$. Then the sequences
$(F^i(X))_{i\in\BN}$ and $(F^{-i}(X))_{i\in\BN}$ fulfill the conditions in
the definition above proving that both ends of $M_f'$ are singly
degenerate.

Before going further recall that homotopic, $\pi_1$--injective simplicial
hyperbolic surfaces can be interpolated by simplicial hyperbolic surfaces.
In particular, we obtain that every point in $M_f'$ is contained in a
$CAT(-1)$ surface. The same holds for every point in a sufficiently small
neighborhood of a singly degenerated end. This is the key observation
leading to the proof of Thurston and Canary's \cite{Canary-covering}
covering theorem.

\begin{CTh}
\hypertarget{CThe}{Let}
$M$ and $N$ be infinite volume hyperbolic 3--manifolds with finitely
generated fundamental group and $\pi \co M\to N$ be a Riemannian covering.
Every simply degenerate end $\CE$ of $M$ has a neighborhood homeomorphic to
$E=S\times[0,\infty)$ such that $\pi(E)=R\times[0,\infty)$ where $R$ is a
closed surface and $\pi\vert_E \co E\to\pi(E)$ is a finite-to-one covering.
\end{CTh}

The \hyperlink{CThe}{Covering Theorem} would be of limited use if there were a third, call it
wild, kind of ends. However, the positive solution of the tameness
conjecture by Agol \cite{Agol} and Calegari--Gabai \cite{Calegari-Gabai},
together with an older but amazingly nice result of Canary
\cite{Canary-tame}, implies that that every end of a hyperbolic 3--manifold
without cusps is either convex-cocompact or simply degenerate; there is an
analogous statement in the presence of cusps.

\begin{sat}\label{tameness}
If $M$ is a hyperbolic 3--manifold with finitely generated fundamental group
then $M$ is homeomorphic to the interior of a compact 3--manifold. In
particular, in the absence of cusps, every end of $M$ is either
convex-cocompact or simply degenerate.
\end{sat}

\fullref{tameness} asserts that in order to prove that a manifold without
cusps is convex-cocompact it suffices to prove that it does not have any
degenerate ends. On the other hand, the \hyperlink{CThe}{Covering Theorem} asserts that
degenerate ends cannot cover infinite volume 3--manifolds in interesting
ways. We obtain for example the following corollary.

\begin{kor}\label{ugly}
Let $M$ be trivial interval bundle or a handlebody and let
$\rho \co \pi_1(M)\to\PSL_2\BC$ be a hyperbolic structure on $M$ such that
$N_\rho=\BH^3/\rho(\pi_1(M))$ has no cusps. If $\Gamma\subset\pi_1(M)$ is a
finitely generated subgroup of infinite index then $\BH^3/\rho(\Gamma)$ is
convex-cocompact.
\end{kor}

Given a sequence of pointed hyperbolic 3--manifolds $(M_i,p_i)$ such that
the injectivity radius $\inj(M_i,p_i)$ of $M_i$ at $p_i$ is uniformly
bounded from below, then it is well-known that we may extract a
geometrically convergent subsequence; say that it is convergent itself.
More precisely, this means that there is some pointed 3--manifold $(M,p)$
such that for every large $R$ and small $\epsilon$, there is some $i_0$
such that for all $i\geq i_0$, there are $(1+\epsilon)$--bi-Lipschitz, base
points preserving, embeddings $\kappa_i^R \co (B_R(p,M),p)\rightarrow(M_i,p_i)$
of the ball $B_R(M,p)$ in $M$ of radius $R$ and center $p$. Taking $R$ and
$\epsilon$ in a suitable way, we obtain better and better embeddings of
larger and larger balls and we will refer in the sequence to these maps as
the {\em almost isometric embeddings\/} provided by geometric convergence. If
$M_i$ is a sequence of hyperbolic 3--manifolds and $M$ is isometric to the
geometric limit of some subsequence of $(M_i,p_i)$ for some choice of base
points $p_i\in M_i$ then we say that $M$ is a geometric limit of the
sequence $(M_i)$. We state here the following useful observation.

\begin{lem}\label{super-ugly}
If $M$ is a geometric limit of a sequence $(M_i)$, $K\subset M$ is a
compact subset such that the image $[\pi_1(K)]$ of $\pi_1(K)$ in $\pi_1(M)$
is convex-cocompact, and $\kappa_i:K\to M_i$ are the almost isometric maps
provided by geometric convergence, then, for all $i$ large enough, the
induced homomorphism
$$(\kappa_i)_* \co [\pi_1(K)]\to \pi_1(M_i)$$
is injective and has convex-cocompact image.
\end{lem}
\begin{proof}
Let $\tilde C$ be a convex, $[\pi_1(K)]$--invariant subset of $\BH^3$ with
$\tilde C/[\pi_1(K)]$ compact and let $C$ be the image of $\tilde C$ in
$M$. We find $i_0$ such that for all $i\ge i_0$ the almost isometric
embeddings given by geometric convergence are defined on $C$ and, moreover,
the lift $\tilde C\to \BH^3$ of the composition of (1) projecting from
$\tilde C$ to $C$, and (2) mapping $C$ into $M_i$ is a quasi-isometry. By
\fullref{cccq}, $\tilde C$ is quasi-isometric to $[\pi_1(K)]$. The claims
follows now applying again \fullref{cccq}.
\end{proof}

It is well-known that a pseudo-Anosov mapping class $f\in\Map(S)$ of a
surface fixes two projective classes of laminations
$\lambda^+,\lambda^-\in\CP\CM\CL(S)$; $\lambda^+$ is the attracting fix
point and $\lambda^-$ the repelling. More precisely, for any essential
simple closed curve $\gamma$ in $S$ we have
$$\lim_{n\to\infty}f^n(\gamma)=\lambda^+, \qua 
\lim_{n\to\infty}f^{-n}(\gamma)=\lambda^-$$
where the limits are taken in $\CP\CM\CL(S)$. If $M_f'$ denotes again the
infinite cyclic cover of the mapping torus $M_f$ then when $\vert n\vert$
becomes large then the geodesic representatives in $M_f'$ of $f^n(\gamma)$
leave every compact set. This implies that the laminations $\lambda^+$ and
$\lambda^-$ are in fact the {\em ending laminations\/} of $M_f'$. In general,
every singly degenerate end has an associated ending lamination
$\lambda_\CE$. More precisely, if $N_\rho$ is a hyperbolic structure on
$M$, $S$ is a component of $\bnd M$ and $\CE$ is the $S$--end of $N_\rho$,
then the ending lamination $\lambda_\CE$ is defined as the limit in the
space of laminations on $S$ of any sequence of simple closed curves
$\gamma_i$, on $S$, whose geodesics representatives $\gamma_i^*$ tend to
the end $\CE$ and are homotopic to $\gamma_i$ within $\CE$. Thurston's
ending lamination conjecture asserts that every hyperbolic structure on a
3--manifold, say for simplicity without cusps, is fully determined by its
{\em ending invariants\/}: the conformal structures associated to the
convex-cocompact ends and the ending lamination associated to the singly
degenerate ends. The ending lamination conjecture has been recently proved
by Minsky \cite{ELC1} and Brock--Canary--Minsky \cite{ELC2}. However, from
our point of view, the method of proof is much more relevant than the
statement itself: the authors prove that given a manifold and ending
invariants, satisfying some necessary conditions, then it is possible to
construct a metric on the manifold, the {\em model\/}, which is bi-Lipschitz
equivalent to any hyperbolic metric on the manifold with the given ending
invariants.

%
%
%
%

\section{Minimal surfaces and Heegaard splittings}\label{sec:minimal}
In this section let $M=(M,\rho)$ be a closed Riemannian 3--manifold, not
necessarily hyperbolic. We will however assume that $M$ is irreducible.
Recall that a surface $F\subset M$ is a minimal surface if it is a critical
point for the area functional. More precisely, if $\CH^2_M(F)$ is the area,
or in other words the two dimensional Hausdorff measure of the surface $F$
in $M$, then $F$ is {\em minimal\/} if for every smooth variation $(F_t)_t$
with $F_0=F$ one has
\begin{equation}\label{eq:area-variation}
\frac d{dt}\CH^2_M(S_t)\vert_{t=0}=0.
\end{equation}
The minimal surface $F$ is said to be {\em stable\/} if again for every
smooth variation the second derivative of the area is positive:
\begin{equation}\label{eq:area-variation2}
\frac {d^2}{dt^2}\CH^2_M(S_t)\vert_{t=0}>0.
\end{equation}
From a more intrinsic point of view, it is well-known that a surface $F$ in
$M$ is minimal if and only if its mean curvature vanishes.

Schoen--Yau \cite{Schoen-Yau} and Sacks--Uhlenbeck \cite{Sacks-Uhlenbeck}
proved that every geometrically incompressible surface $S$ in $M$ is
homotopic to a stable minimal surface $F$. Later, Freedman--Hass--Scott
\cite{FHS} proved that in fact $F$ is embedded and hence that, by a result
of Waldhausen \cite{Waldhausen} $S$ is isotopic to a connected component of
the boundary of a regular neighborhood of $F$. Summing up one has the 
following theorem.

\begin{sat}\label{thm:incompressible-minimal}
Let $S$ be a geometrically incompressible surface in $M$. Then there is a
stable minimal surface $F$ such that $S$ is either isotopic to $F$ or to
the boundary $\bnd\CN(F)$ of a regular neighborhood of $F$.
\end{sat}

\fullref{thm:incompressible-minimal} concludes the discussion about
existence of minimal surfaces as long as one is only interested into those
surfaces which are geometrically incompressible. We turn now our attention
to surfaces which are geometrically compressible.

Not every compressible surface in $M$ needs to be isotopic to a minimal
surface. In fact, the following beautiful theorem of Lawson
\cite{Lawson-unknotted} asserts that for example every minimal surface $F$
in the round 3--sphere $\BS^3$ is a Heegaard surface.

\begin{sat}\label{thm:minimal in sphere}
Assume that $M$ has positive Ricci-curvature $\Ric(M)>0$ and let $F$ be a
closed embedded minimal surface. Then, $M\setminus F$ consists of one or
two handlebodies.
\end{sat}

However, there is a way to associate to every surface in $M$ a (possibly
empty) minimal surface. The idea is to consider the set $\FI(S)$ of all
possible surfaces in $M$ isotopic to $S$ and try to minimize area. If $S_i$
is a sequence in $\FI(S)$ such that
$$\lim_i\CH^2_M(S_i)=\inf\{\CH^2_M(S')\vert S'\in\FI(S)\}$$
then one can try to extract a limit of the surfaces $S_i$ hoping that it
will be a minimal surface. However, it is unclear which topology should one
consider. The usual approach is to consider $S_i$ as a varifold. A {\em
varifold\/} is a Radon measure on the Grassmannian $G^2(M)$ of
two-dimensional planes in $TM$. For all $i$, the inclusion of the surface
$S_i$ lifts to an inclusion
$$S_i\to G^2(M)$$
obtained by sending $x\in S_i$ to the plane $T_xS_i\in G^2(M)$. We obtain
now a measure on $G^2(M)$ by pushing forward the Hausdorff measure, ie
the area, of $S_i$. Observe that the total measure of the obtained varifold
coincides with the area $\CH_M^2(S_i)$ of $S_i$. In particular, the
sequence $(S_i)$, having area uniformly bounded from above, has a
convergent subsequence, say the whole sequence, in the space of varifolds.
Let $F=\lim_iS_i$ be its limit. The hope now is that $F$ is a varifold
induced by an embedded minimal surface. Again in the language of varifolds,
$F$ is a so-called {\em stationary varifold\/} and Allard's \cite{Allard}
regularity theory asserts that it is induced by a countable collection of
minimal surfaces. In fact, using the approach that we just sketched,
Meeks--Simon--Yau proved the following theorem.

\begin{sat}[Meeks--Simon--Yau \cite{MSY}]\label{MSY}
Let $S$ an embedded surface in $M$ and assume that
$$\inf\{\CH^2(S')\vert S'\in\FI(S)\}>0.$$
Then there is a minimizing sequence in $\FI(S)$ converging to a varifold
$V$, a properly embedded minimal surface $F$ in $M$ with components
$F_1,\dots,F_k$ and a collection of positive integers $m_1,\dots,m_k$ such
that $V=\sum m_iF_i$.
\end{sat}

\fullref{MSY} applies also if $S$ is a properly embedded surface in a
manifold with mean-convex, for instance minimal, boundary.

\fullref{thm:minimal in sphere} implies that the minimal surface $F$
provided by \fullref{MSY} is, in general, not isotopic to the surface $S$
we started with. Moreover, since the notion of convergence is quite weak,
it seems hopeless to try to relate the topology of both surfaces. However,
Meeks--Simon--Yau \cite{MSY} show, during the proof of \fullref{MSY}, that
$F$ arises from $S$ through surgery.

\begin{bem}
Meeks--Simon--Yau say that the minimal surface $F$ arises from $S$ by
$\gamma$--{\em convergence\/} but this is exactly what we call surgery.
\end{bem}

\fullref{MSY}, being beautiful as it is, can unfortunately not be used if
$S$ is a Heegaard surface. Namely, if $S$ is a Heegaard surface in $M$ then
there is a sequence of surfaces $(S_i)$ isotopic to $S$ such that
$\lim_i\CH^2_M(S_i)=0$. For the sake of comparison, every simple closed
curve in the round sphere $\BS^2$ is isotopic to curves with arbitrarily
short length. The comparison with curves in the sphere is not as
far-fetched as it may seem. From this point of view, searching from minimal
surfaces amounts to prove that the sphere has a closed geodesic for every
Riemannian metric. That this is the case is an old result due to Birkhoff.

\begin{sat}[Birkhoff]\label{Birkhoff}
If $\rho$ is a Riemannian metric on $\BS^2$, then there is a closed
non-constant geodesic in $(\BS^2,\rho)$.
\end{sat}

The idea of the proof of Birkhoff's theorem is as follows. Fix $\rho$ a
Riemannian metric on $\BS^2$ and let
$$f \co \BS^1\times[0,1]\to\BS^2, \qua f(\theta,t)=f_t(\theta)$$
be a smooth map with $f_0$ and $f_1$ constant and such that $f$ represents
a non-trivial element in $\pi_2(\BS^2)$. For any $g$ homotopic to $f$ let
$$E(g)=\max\{l_\rho(g_t)\vert t\in[0,1]\}$$
be the length of the longest of the curves $g_t$. Observe that since $g$ is
not homotopically trivial $E(g)$ is bounded from below by the injecvity
radius of $(\BS^2,\rho)$. Choose then a sequence $(g^i)$ for maps homotopic
to $f$ such that
$$\lim_i E(g^i)=\inf\{E(g)\vert g \text{ homotopic to } f\}.$$
One proves that there is a {\em minimax\/} sequence $(t^i)$ with
$t^i\in[0,1]$ such that $E(g_i)=l_\rho(g^i_{t^i})$ and such that the curves
$g^i_{t^i}$ converge, when parametrized by arc-length to a non-constant
geodesic in $(\BS^2,\rho)$.

The strategy of the proof of Birkhoff's theorem was used in the late 70s
by Pitts \cite{Pitts} who proved that every closed $n$--manifold with $n\le
6$ contains an embedded minimal submanifold of codimension 1 (see also
Schoen--Simon \cite{Schoen-Simon} for $n=7$). We describe briefly his proof in the
setting of 3--manifolds. The starting point is to consider a Heegaard
surface $S$ in $M$. By definition, the surface $S$ divides $M$ into two
handlebodies. In particular, there is a map
\begin{equation}\label{eq:sweep-out}
f \co (S\times[0,1],S\times\{0,1\})\to (M,f(S\times\{0,1\})),\qua f(x,t)=f^t(x)
\end{equation}
with positive relative degree, such that for $t\in(0,1)$ the map $f^t \co S\to
M$ is an embedding isotopic to the original embedding $S\hookrightarrow M$
and such that $f^0(S)$ and $f^1(S)$ are graphs. Such a map as in
\eqref{eq:sweep-out} is said to be a {\em sweep-out\/} of $M$. Given a
sweep-out $f$ one considers $E(f)$ to be the maximal area of the surfaces
$f^t(S)$. Pitts proves that there is a minimizing sequences $(f_i)$ of
sweep-outs and an associated {\em minimax\/} sequence $t^i$ with
$E(f^i)=\CH^2_M(f_{t^i}(S))$ and such that the surfaces $f_{t^i}(S)$
converge as varifolds to an embedded minimal surface $F$; perhaps with
multiplicity.

\begin{sat}[Pitts]\label{Pitts}
Every closed Riemannian 3--manifold contains an embedded minimal surface.
\end{sat}

Pitts' proof is, at least for non-experts like the author of this note,
difficult to read. However, there is an amazingly readable proof due to
Colding--de Lellis \cite{Colding-deLellis}. In fact, the main technical
difficulties can be by-passed, and this is what these authors do, by using
Meeks--Simon--Yau's \fullref{MSY}. In fact, as it is the case with the
Meeks--Simon--Yau theorem, \fullref{Pitts} remains true for compact
3--manifolds with mean-convex boundary.

\fullref{Pitts} settles the question of existence of minimal surfaces in
3--manifolds. Unfortunately, it does not say anything about the relation
between the Heegaard surface $S$ we started with and the obtained minimal
surface $F$. In fact, Colding--de Lellis announce in their paper that in a
following paper they are going to prove that the genus does not increase.
The concept of convergence of varifolds is so weak that  this could well
happen. However, in the early 80s, Pitts and Rubinstein affirmed something
much stronger: they claimed that $F$ is not stable and arises from $S$ by
surgery. This was of the greatest importance in the particular case that
the Heegaard surface $S$ is assumed to be strongly irreducible.

By \fullref{strongly}, the assumption that the Heegaard surface $S$ is
strongly irreducible implies that every surface $S'$ which arises from $S$
by surgery is either isotopic to $S$ or of one of the following two kinds:
\begin{itemize}
\item[(A)] Either $S'$ is obtained from $S$ by suturing along disks which
are all at the same side, or
\item[(B)] $S$ is isotopic to the surface obtained from the boundary of a
regular neighborhood of $S'$ by attaching a vertical handle.
\end{itemize}
In particular, if in the setting of Pitts' theorem we assume that $S$ is
strongly irreducible we obtain that this alternative holds for the minimal
surface $F$. In fact, more can be said. If we are in case (A) then $F$
bounds a handlebody $H$ in $M$ such that the surface $S$ is isotopic to a
strongly irreducible Heegaard surface in the manifold with boundary
$M\setminus H$. The boundary of $M\setminus H$ is minimal an
incompressible. In particular, $F=\bnd M\setminus H$ is isotopic to some
stable minimal surface $F'$ in $M\setminus H$ parallel to $\bnd M$. Observe
that $F\neq F'$ because one of them is stable and the other isn't. The
stable minimal surface $F'$ bounds in $M$ some submanifold $M_1$ isotopic
in $M$ to $M\setminus H$, in particular the original Heegaard surface $S$
induces a Heegaard splitting of $M_1$. The boundary of $M_1$ is minimal and
hence mean-convex. In particular, the method of proof of \fullref{Pitts}
applies and yields an unstable minimal surface $F_1$ in $M_1$ obtained from
$S$ by surgery. Again we are either in case (B) above, or $F_1$ is isotopic
to $S$ within $M_1$ and hence within $M$, or we can repeat this process.

\begin{figure}[ht!]
\begin{center}
\includegraphics[scale=0.4,angle=0]{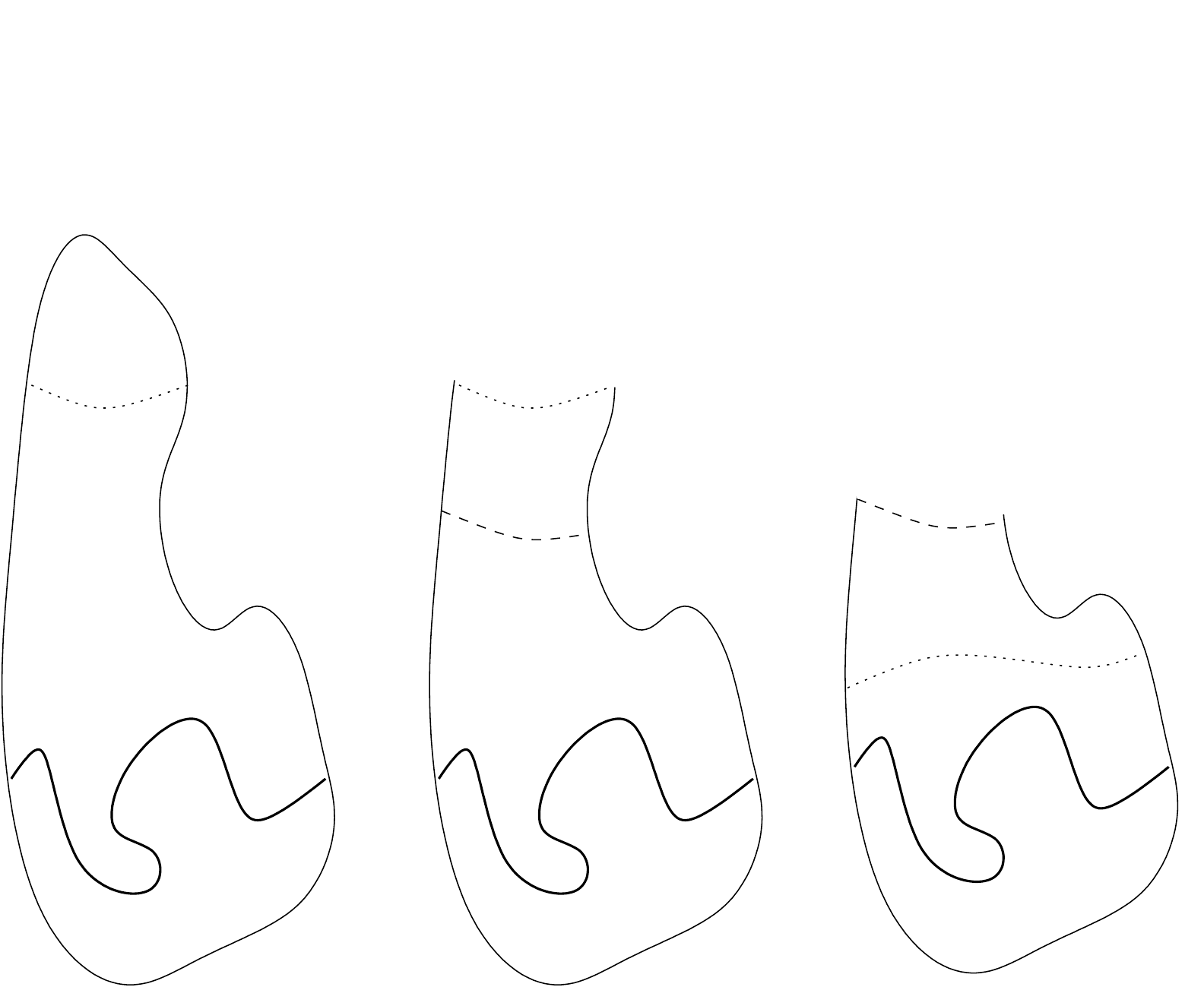}
\end{center}
\caption{Proof of \fullref{Pitts-Rubinstein}: The thick line is the
original surface; the short dotted line is the first minimax surface; the
dashed line is the least area surface obtained from the first minimax
surface; the long dotted line is the second minimax surface.}
\label{fig:3}
\end{figure}

If this process goes for ever, we obtain a sequences of disjoint embedded
minimal surfaces in $M$ with genus less than that of $S$. This means that
the metric in $M$ is not {\em bumpy\/}. However, if $M$ is not bumpy, then we
can use a result of White \cite{White-bumpy} and perturb it slightly so
that it becomes bumpy. It follows from the above that for any such
perturbation the process ends and we obtain a minimal surface which is
either as in (B) or actually isotopic to the Heegaard surface $S$. Taking a
sequence of smaller and smaller perturbations and passing to a limit we
obtain a minimal surface $F$ in $M$, with respect to the original metric,
which is either isotopic to $S$ or such that we are in case (B) above. In
other words we have the following theorem.

\begin{sat}[Pitts-Rubinstein]\label{Pitts-Rubinstein}
If $S$ is a strongly irreducible Heegaard surface in a closed 3--manifold
then there is a minimal surface $F$ such that $S$ is either isotopic to $F$
or to the surface obtained from the boundary of a regular neighborhood of
$F$ by attaching a vertical 1--handle.
\end{sat}

Unfortunately, Pitts and Rubinstein never wrote the proof of
\fullref{Pitts-Rubinstein} above and it seems unlikely that they are ever
going to do so. The most precise version known to the author is a sketch of
the proof due to Rubinstein \cite{Rubinstein}. This lack of written proof
has made doubtful if one could use \fullref{Pitts-Rubinstein} safely or
not. However, all that is left is to prove that the minimal surface
provided in the proof of Pitts' \fullref{Pitts} is unstable and obtained
from $S$ by surgery. The author of this note has written a proof
\cite{Juan-book} and is working on a longer text, perhaps a book,
explaining it and some applications of the Pitts--Rubinstein theorem.

Before concluding this section we should remember that one can combine
\fullref{all-generalized} and \fullref{Pitts-Rubinstein} as follows: Given
a Heegaard surface we first destabilize as far as possible, then we obtain
using \fullref{all-generalized} a generalized Heegaard surface
$(\Sigma_I,\Sigma_H)$. The surface $\Sigma_I$ is incompressible and hence
can be made minimal by \fullref{thm:incompressible-minimal}; now the
surface $\Sigma_H$ can be made minimal using \fullref{Pitts-Rubinstein}.

\section{Using geometric means to determine the Heegaard
genus}\label{sec:genus}
In this section we will show how minimal surfaces can be used to compute
the Heegaard genus of some manifolds. Most, if not all, of the results we
discuss here can be proved using purely topological arguments but, in the
opinion of the author, the geometric proofs are beautiful.

We start considering the mapping torus $M_\phi$ of a pseudo-Anosov mapping
class $\phi\in\Map(\Sigma_g)$ on a closed surface of genus $g$; compare
with \eqref{eq:map-torus}. It is well-known that $M_\phi$ admits a weakly
reducible Heegaard splitting of genus $2g+1$. In particular we have the
following bound for the Heegaard genus
$$g(M_\phi)\le 2g+1.$$
There are manifolds which admit different descriptions as a mapping torus.
In particular, we cannot expect that equality always holds. However,
equality is to be expected if monodromy map $\phi$ is complicated enough.

\begin{sat}\label{genus-mapping1}
Let $\Sigma_g$ be a closed surface of genus $g$ and $\phi\in\Map(\Sigma_g)$
a pseudo-Anosov mapping class. Then there is $n_\phi>0$ such that for all
$n\ge n_\phi$ one has $g(M_{\phi^n})=2g+1$. Moreover, for every such $n$
there is, up to isotopy a unique Heegaard splitting of $M_{\phi^n}$ of
genus $2g+1$.
\end{sat}

We sketch now the proof of \fullref{genus-mapping1}. More precisely, we
will prove that, for large $n$, there is no strongly irreducible Heegaard
splitting of $M_{\phi^n}$ of genus at most $2g+1$. The general case
follows, after some work, using \fullref{all-generalized}.

Seeking a contradiction, assume that $M_{\phi^n}$ admits a strongly
irreducible splitting with at most genus $2g+1$. Then, endowing
$M_{\phi^n}$ with its hyperbolic metric, we obtain from
\fullref{Pitts-Rubinstein} that $M_{\phi^n}$ contains a minimal surface $F$
of at most genus $2g+1$ and such that every component of
$M_{\phi^n}\setminus F$ is a handlebody. In particular, $F$ intersects
every copy of the fiber $\Sigma_g$ since the later is incompressible and a
handlebody does not contain any incompressible surfaces. For all $n$ the
manifold $M_{\phi^n}$ covers the manifold $M_\phi$. In particular, we have
first the following lower bound for the injectivity radius
$$\inj(M_{\phi^n})\ge\inj(M_\phi)$$
and secondly that increasing $n$ we can find two copies of the fiber which
are at arbitrary large distances. On the other hand, the bounded diameter
lemma for minimal surfaces below shows that the diameter of a minimal
surface in a hyperbolic 3--manifold is bounded from above only in terms of
its genus and of the injectivity radius of the manifold. This shows that if
$n$ is large the minimal surface $F$ cannot exist.

\medskip
{\bf Bounded diameter lemma for minimal surfaces (first version)}\qua\sl
Let $F$ be a connected minimal surface in a hyperbolic 3--manifold $M$ with
at least injectivity radius $\epsilon$. Then we have
$$\diam(F)\le \frac{4\vert\chi(F)\vert}\epsilon+2\epsilon$$
where $\diam(F)$ is the diameter of $F$ in $M$.\rm

\medskip
\begin{proof}
The motonicity formula (Colding--Minicozzi \cite{Colding-Minicozzi-book}) 
asserts that for every
point $x\in F$ we have $\CH^2_M(F\cap B_x(M,\epsilon))\ge\pi\epsilon^2$
where $B_x(M,\epsilon)$ is the ball in $M$ centered at $x$ and with radius
$\epsilon$.  If $F$ has diameter $D$ in $M$ we can find at least $\frac
D{2\epsilon}-1$ points which are at distance at least $\epsilon$ from each
other. On the other hand, the curvature of $F$ is bounded from above by
$-1$ and hence the total area is bounded by $\CH^2_M(F)\le
2\pi\vert\chi(F)\vert$. In particular we obtain that
$$\left(\frac D{2\epsilon}-1\right)\pi\epsilon^2\le 2\pi\vert\chi(F)\vert.$$
This concludes the proof.
\end{proof}

We stated this as a {\em first version\/} because in some sense the role of
the injectivity radius of $M$ is disappointing. However, it is not
difficult to construct hyperbolic 3--manifolds containing minimal surfaces
of say genus $2$ with arbitrarily large diameter. In order to by-pass this
difficulty we define, following Thurston, for some $\epsilon$ positive the
length of a curve $\gamma$ relative to the $\epsilon$--thin part
$M^{<\epsilon}$ of $M$ to be the length of the intersection of $\gamma$
with the set of points in $M$ with injectivity radius at least $\epsilon$.
Then, the distance $d_{\rel M^{<\epsilon}}(x, y)$ of two points $x,y\in M$
relative to the the $\epsilon$-thin part is the infimum of the lengths
relative to the $\epsilon$-thin part of paths joining $x$ and $y$. Using
this pseudo-distance we obtain with essentially the same proof the
following final version of the bounded diameter lemma.

\medskip{\bf Bounded diameter lemma for minimal surfaces}\qua\sl
Let $F$ be a connected minimal surface in a hyperbolic 3--manifold $M$ and
let $\mu>0$ be the Margulis constant. Then we have
$$\diam_{\rel M^{<\mu}}(F)\le \frac{8\vert\chi(F)\vert}\mu$$
where $\diam_{\rel M^{<\mu}}(F)$ is the diameter of $F$ in $M$ with respect
to $d_{\rel M^{<\mu}}$.
\rm

\medskip
The bounded diameter lemma, together with the argument used in the proof of
\fullref{genus-mapping1} shows now that whenever $M_\phi$ is the mapping
torus of a pseudo-Anosov mapping class on the surface $\Sigma_g$ and such
that $M_\phi$ contains two fibers which are at least at distance
$\frac{8(2g+4)}\mu$  with respect to $d_{\rel M^{<\mu}}$ then $M_\phi$ does
not have strongly irreducible Heegaard splittings of genus less than
$2g+1$. Moreover, one gets as above that $M_\phi$ has genus $2g+1$ and, up
to isotopy, a single Heegaard splitting of minimal genus. In particular, in
order to generalize \fullref{genus-mapping1}, it suffices to give
conditions on the monodromy $\phi$ ensuring that the mapping torus $M_\phi$
contains fibers at large distance. It follows for example from the work of
Minsky \cite{Minsky} that this is the case if the translation length of
$\phi$ in the curve complex $\CC(\Sigma_g)$ is large enough. In particular
we have the following theorem.

\begin{sat}\label{genus-mapping2}
For every $g$ there is $D_g>0$ such that the following holds: If $\phi$ is
a pseudo-Anosov mapping class on $\Sigma_g$ with at least translation
length $D_g$ in the curve complex $\CC(\Sigma_g)$ of $\Sigma_g$, then the
mapping torus $M_\phi$ has Heegaard genus $g(M_\phi)=2g+1$ and there is, up
to isotopy, a unique minimal genus Heegaard splitting.
\end{sat}

Observe that whenever $\phi$ is pseudo-Anosov, the translation lengths of
$\phi^n$ in $\CC(\Sigma_g)$ tends to $\infty$ when $n$ becomes large. In
other word, \fullref{genus-mapping1} follows from \fullref{genus-mapping2}.

An other result in the same spirit is due to Lackenby
\cite{Lackenby-genus}, who proved that whenever $M_1$ and $M_2$ are
compact, irreducible, atoroidal 3--manifolds with incompressible and
acylindrical homeomorphic connected boundaries $\bnd M_1=\bnd M_2$ of at
least genus $2$ and $\phi\in\Map(\bnd M_1)$ is a pseudo-Anosov mapping
class then the manifold $M_1\cup_{\phi^n}M_2$ obtained by gluing $M_1$ and
$M_2$ via $\phi^n$ has Heegaard genus $g(M_1)+g(M_2)-g(\bnd M_1)$. In this
setting the key point is again that if $n$ is large then the manifold
$M_1\cup_{\phi^n}M_2$ contains two surfaces isopic to the gluing surface
and which are at large distance. And again there is a generalization
involving the curve complex.

\begin{sat}[Souto \cite{Souto-genus}]\label{simple-genus}
Let $M_1$ and $M_2$ be compact, irreducible, atoroidal 3--manifolds with
incompressible and acylindrical homeomorphic connected boundaries $\bnd
M_1=\bnd M_2$ of genus at least two and fix an essential simple closed
curve $\alpha$ in $\bnd M_1$. Then there is a constant $D$ such that every
minimal genus Heegaard splitting of $M_1\cup_\phi M_2$ is constructed
amalgamating splittings of $M_1$ and $M_2$ and hence
$$g(M_1\cup_\phi M_2)=g(M_1)+g(M_2)-g(\bnd M_1)$$
for every diffeomorphism $\phi \co \bnd M_1\to\bnd M_1$ with $d_{\CC(\bnd
M_1)}(\phi(\alpha),\alpha)\ge D$.
\end{sat}

Lackenby's argument in fact uses minimal surfaces and it was reading his
paper when the author of this note became interested in the relation
between minimal surfaces and Heegaard splittings. Lackenby's paper is also
beautifully written.

A different result, the oldest of the ones presented in this section, and
also proved using minimal surfaces, involves the Heegaard genus of those
hyperbolic 3--manifolds obtained by Dehn-filling a finite volume manifold
with cusps. Recall that every complete, non-compact, orientable complete
hyperbolic manifold $M$ with finite volume is homeomorphic to the interior
of a compact manifold $\wbar M$ with torus boundary. For simplicity we will
assume that $\bnd \wbar M$ has only one component; we say that $M$ has a
single cusp. Identifying $\bnd\wbar M$ with the boundary of a solid torus
$\BD^2\times\BS^1$ via some map $\phi$ and gluing both $\wbar M$ and
$\BD^2\times\BS^1$ via this identification we obtain a closed 3--manifold
$M_\phi$ which is said to arise from $M$ by Dehn-filling. In fact, the
homeomorphism type of $M_\phi$ depends only on the homotopy class in
$\bnd\wbar M$ of the meridian of the attached solid torus. In other words,
for every essential simple closed curve $\gamma$ in $\bnd M$ there is a
unique 3--manifold $M_\gamma$ obtained by Dehn-filling along $\gamma$.
Thurston's beautiful Dehn filling theorem asserts that for all but finitely
many $\gamma $ the manifold $M_\gamma$ admits a hyperbolic metric. An
extension due to Hodgson--Kerckhoff \cite{Hodgson-Kerckhoff} of this result
asserts that the number of exceptions is in fact bounded independently of
the manifold $M$.

By construction, there is a natural embedding $M\hookrightarrow M_\gamma$
and it is easy to see that every Heegaard surface of $M$ is, via this
embedding, also a Heegaard surface of $M_\gamma$ for every Dehn-filling of
$M$. This proves that
$$g(M_\gamma)\le g(M)$$
for all $M_\gamma$. It is not difficult to construct examples which show
that there are finite volume hyperbolic manifolds admitting infinitely many
Dehn-fillings for which equality does not hold. However, in some sense,
equality holds for most Dehn-fillings of $M$. More precisely, identifying
the set $\CS(\bnd\wbar M)$ of homotopy class of essential simple curves in
$\bnd\wbar M$ with the set of vertices of the Farey graph one obtains a
distance on $\CS(\bnd\wbar M)$. Using this distance one has the following
theorem.

\begin{sat}[Moriah--Rubinstein
\cite{Moriah-Rubinstein}]\label{Moriah-Rubinstein}
Let $M$ be a complete, oriented finite volume hyperbolic manifold with a
cusp. Then there is a bounded set $K$ in $\CS(\bnd\wbar M)$ with
$g(M_\gamma)=g(M)$ for every $\gamma\notin K$.
\end{sat}

Before concluding this section observe that the Farey graph is not locally
compact and hence bounded sets may be infinite. Observe also that a purely
topological proof of \fullref{Moriah-Rubinstein} is due to Rieck--Sedgwick
\cite{Rieck-Sedgwick}.

\section{Generators of the fundamental group and carrier
graphs}\label{sec:carrier}
Let $M$ be a closed hyperbolic, or more generally, negatively curved
3--manifold. In this section we relate generating sets, or more precisely
Nielsen equivalence classes of generating sets, of $\pi_1(M)$ to some
graphs in $M$ with nice geometric properties.

Recall that two (ordered) generating sets $\CS=(g_1,\dots,g_r)$ and
$\CS'=(g_1',\dots,g_r')$ of a group are {\em Nielsen equivalent\/} if they
belong to the same class of the equivalence relation generated by the
following three moves:
$$\begin{array}{ll}
\text{Inversion of } g_i & \left\{\begin{array}{ll}g_i'=g_i^{-1} & \\
g_k'=g_k & k\neq i.\end{array}\right. \\
\text{Permutation of } g_i \text{ and } g_j \text{ with } i\neq j&
\left\{\begin{array}{ll}g_i'=g_j & \\ g_j'=g_i & \\ g_k'=g_k & k\neq
i,j.\end{array}\right. \\
\text{Twist of } g_i \text{ by } g_j \text{ with } i\neq j &
\left\{\begin{array}{ll}g_i'=g_ig_j & \\ g_k'=g_k & k\neq
i.\end{array}\right.
\end{array}$$

To every Nielsen equivalence class of generators of $\pi_1(M)$ one can
associate an equivalence class of carrier graphs.

\begin{defi*}
A continuous map $f \co X\to M$ of a connected graph $X$ into a hyperbolic
3--manifold $M$ is a {\em carrier graph\/} if the induced homomorphism
$f_* \co \pi_1(X)\to\pi_1(M)$ is surjective. Two carrier graphs $f \co X\to M$ and
$g \co Y\to M$ are {\em equivalent\/} if there is a homotopy equivalence 
$h \co X\to Y$ such that $f$ and $g\circ h$ are free homotopic.
\end{defi*}

Given a generating set $\CS=(g_1,\dots,g_r)$ of $\pi_1(M)$ let $\BF_\CS$ be
the free non-abelian group generated by the set $\CS$,
$\phi_\CS \co \BF_\CS\to\pi_1(M)$ the homomorphism given by mapping the free
basis $\CS\subset\BF_\CS$ to the generating set $\CS\subset\pi_1(M)$ and
$X_\CS$ a graph with $\pi_1(X_\CS)=\BF_\CS$. The surjective homomorphism
$\phi_\CS \co \BF_\CS\to\pi_1(M)$ determines a free homotopy class of maps
$f_\CS \co X_\CS\to M$, ie a carrier graph, and any two carrier graphs
obtained in this way are equivalent. The so determined equivalence class is
said to be the {\em equivalence class of carrier graphs associated to
$\CS$\/}.

\begin{lem}\label{Nielsen}
Let $\CS$ and $\CS'$ be finite generating sets of $\pi_1(M)$ with the same
cardinality. Then the following are equivalent.
\begin{enumerate}
\item $\CS$ and $\CS'$ are Nielsen equivalent.
\item There is a free basis $\wbar\CS$ of $\BF_{\CS'}$ with
$\CS=\phi_{\CS'}(\wbar\CS)$.
\item There is an isomorphism $\psi \co \BF_\CS\to\BF_{\CS'}$ with
$\phi_\CS=\phi_{\CS'}\circ\psi$.
\item $\CS$ and $\CS'$ have the same associated equivalence classes of
carrier graphs.
\end{enumerate}
\end{lem}

We will only consider carrier graphs $f \co X\to M$ with
$\rank(\pi_1(X))=\rank(\pi_1(M))$. Equivalently we only consider generating
sets with minimal cardinality.

If $f \co X\to M$ is a carrier then let $X^{(0)}$ be the set of vertices of $X$
and $X^{(1)}$ that of edges. The {\em length\/} of a carrier graph $f \co X\to M$
is defined as the sum of the lengths of the images of the edges
$$l_{f \co X\to M}(X)=\sum_{e\in X^{(1)}}l_M(f(e)).$$
A {\em minimal length\/} carrier graph is a carrier graph $f \co X\to M$ with
$$l_{f \co X\to M}(X)\le l_{f' \co X'\to M}(X')$$
for every other equivalent carrier graph $f':X'\to M$. The existence of
minimal length carrier graphs follows from the Arzela--Ascoli theorem if $M$
is closed and in fact one has the following lemma.

\begin{lem}{\rm White \cite[Section 2]{White}} \qua \label{minimal}
If $M$ is a closed hyperbolic 3--manifold, then there is a minimal length
carrier graph $f \co X\to M$. Moreover, every such minimal length carrier graph
is trivalent, hence it has $3(\rank(M)-1)$ edges, the image in $M$ of its
edges are geodesic segments, the angle between any two adjacent edges is
$\frac{2\pi}3$ and every simple closed path in $X$ represents a non-trivial
element in $\pi_1(M)$.
\end{lem}

It is not difficult to see that if $M_{\phi^n}$ is the mapping torus of a
high power of a pseudo-Anosov mapping class then every minimal length
carrier graph, with minimal cardinality, has huge total length; in
particular it contains some large edge. However, the following simple lemma
asserts that there is a universal upper bound for the length of the
shortest edge in a carrier graph in closed hyperbolic 3--manifold.

\begin{lem}\label{lem:short-edge}
There is some positive $L$ such that every minimal length carrier graph in
a closed hyperbolic 3--manifold has an edge shorter than $L$.
\end{lem}

\begin{proof}
Let $M=\BH^3/\Gamma$ be a hyperbolic 3--manifold and assume that there is a
minimal carrier graph $f \co X\to M$ consist of only extremely long edges.
Denote by $\tilde f \co \tilde X\to\BH^3$ the lift of $f$ to a map between the
universal covers. The image under $\tilde f$ of every monotonous
bi-infinite path in $\tilde X$ consists of extremely long geodesic segments
joined at corners with angle $\frac{2\pi}3$. In particular, every such path
is a quasi-geodesic and $\tilde f$ is a quasi-geodesic embedding, implying
that the homomorphism $f_* \co \pi_1(X)\to\pi_1(M)$ is injective. Hence
$\pi_1(M)$ is free contradicting the assumption that $M$ is closed.
\end{proof}

\fullref{lem:short-edge} for itself is of little use; the point is that it
has some grown up relatives which, in some sense made precise below, allow
to decompose carrier graphs into short pieces. For instance, White
\cite{White} proved that again there is some positive constant $L$ such
that every carrier graph in a closed hyperbolic 3--manifold admits a circuit
shorter than $L$. In particular he obtained the following beautiful result.

\begin{sat}{\rm White \cite{White}} \qua \label{white-inj}
For every integer $r$ there is a positive constant $L$ with $\inj(M)\le L$
for every closed hyperbolic 3--manifold whose fundamental group has at most
rank $r$.
\end{sat}

Unfortunately, White's observation is the end if one takes the most naive
point of view: there are examples showing that every subgraph with
non-abelian fundamental group can be made as long as one wishes or fears.
However, the idea is still to obtain a exhaustion of every carrier graph by
subgraphs which in some sense are short. The solution is to consider the
length of a carrier graph relative to the convex-hull of a subgraph.

However, before making this more precise we need some more notation. If
$f \co X\to M$ is a carrier graph and $Y\subset X$ is a subgraph then let
$Y^{(0)}$ be again the set of vertices and $Y^{(1)}$ that of edges; let
also $\pi_X \co \tilde X\to X$ be the universal covering of $X$ and $\tilde
f \co \tilde X\to\BH^3$ a fixed lift of $f$ to a map between the universal
coverings of $X$ and $M$. If $Y\subset X$ is a connected subgraph of $X$
then every connected component of $\pi_X^{-1}(Y)$ can be identified with
the universal cover of $Y$. Given such a component $\tilde Y$ of
$\pi_X^{-1}(Y)$ let $G(\tilde Y)\subset\pi_1(X)$ be the group of all
covering transformations of $\pi_X \co \tilde X\to X$ preserving $\tilde Y$;
$G(\tilde Y)$ is isomorphic to $\pi_1(Y)$. Denote by $\Gamma_{\tilde Y}$
the image of $G(\tilde Y)$ under the homomorphism
$f_* \co \pi_1(X)\to\pi_1(M)$.

If $Y$ is a connected subgraph of a carrier graph $f \co X\to M$ and $\tilde Y$
is a component of $\pi_X^{-1}(Y)$ we define the {\em thick convex-hull\/}
$TCH(\tilde Y)$ as follows.

\begin{defi*}
The {\em thick convex-hull\/} $TCH(\tilde Y)$ of a component $\tilde Y$ of
$\pi_X^{-1}(Y)$ is the smallest closed convex subset of $\BH^3$ containing
$\tilde f(\tilde Y)$ and with
$$d_{\BH^3}(x,\gamma x)\ge 1$$
for all non-trivial $\gamma\in\Gamma_{\tilde Y}$ and $x\notin TCH(\tilde
Y)$.
\end{defi*}

The thick convex-hull is unique because intersection of convex subsets is
convex and uniqueness implies that $TCH(\tilde Y)$ is invariant under
$\Gamma_{\tilde Y}$ and in particular it contains the convex-hull of the
limit set of $\Gamma_{\tilde Y}$. However, there are several reasons for
introducing the thick convex-hull $TCH(\tilde Y)$ instead of working
directly with the convex-hull of the limit set of $\Gamma_{\tilde Y}$. For
example we want to avoid treating differently the case that $Y$ is a tree.

We are now ready to formally define the length of a carrier graph $f \co X\to
M$ relative to a subgraph $Y$ with $X^{(0)}\subset Y$. If $e\in
X^{(1)}\setminus Y^{(1)}$ is an edge which is not contained in $Y$ and
$\tilde e$ is a lift of $e$ to the universal cover $\tilde X$ of $X$ then
the vertices of $\tilde e$ are contained in two different components
$\tilde Y_1$ and $\tilde Y_2$ of $\pi_X^{-1}(Y)$. We define the length of
$\tilde e$ relative to $\pi_X^{-1}(Y)$ to be the length of the part of
$\tilde f(\tilde e)$ which is disjoint of the union
$$TCH(\tilde Y_1)\cup TCH(\tilde Y_2)$$
of the thick convex-hulls of $\tilde Y_1$ and $\tilde Y_2$. If $\tilde e'$
is a second lift of $e$ to $\tilde X$ then both $\tilde e$ and $\tilde e'$
have the same length relative to $\pi_X^{-1}(Y)$. In particular, the
relative length with respect to $Y$
$$l_{f \co X\to M,\rel(Y)}(e)$$
of the edge $e$ is well-defined.

If $Z\subset X$ is a second subgraph with $Y\subset Z$ then we define the
length of $Z$ relative to $Y$ to be the sum of the relative lengths of all
the edges contained in $Z$ but not in $Y$:
$$l_{f \co X\to M,\rel(Y)}(Z)=\sum_{e\in Z^{(1)}\setminus Y^{(1)}}l_{f \co X\to
M,\rel(Y)}(e).$$
Observe that $l_{f \co X\to M,\rel(X^{(0)})}(X)= l_{f \co X\to M}(X)$.

The most important observation is the following proposition.

\begin{prop}\label{bounded-chains}
There is $L$ such that whenever $M$ is a closed hyperbolic $3$--manifold and
$f \co X\to M$ is a minimal length carrier graph then there is a chain of
subgraphs
$$X^{(0)}=Y_0\subset Y_1\subset\dots\subset Y_k=X$$
with $l_{f \co X\to M,\rel(Y_{i-1})}(Y_i)\le L$ for all $i=1,\dots,k$.
\end{prop}

The idea behind \fullref{bounded-chains} is that in the proof of
\fullref{lem:short-edge} one can replace the vertices of $\tilde X$ by
convex subsets. See Souto \cite{rankII} for details.

\section{The rank of the fundamental group of simple complicated mapping
tori}\label{sec:rank1}
Recall that the {\em rank\/} of a finitely generated group $G$ is the minimal
number of elements needed to generate $G$. While in general the rank of
even a hyperbolic group is not computable (Rips \cite{Rips}), the situation
changes if one is interested in those groups arising as the fundamental
group of a closed 3--manifold.

\begin{sat}{\rm Kapovich--Weidmann
\cite{Kapovich-Weidmann-rank}} \qua\label{Kapovich-Weidmann}
There exists an algorithm which, given a finite presentation of the
fundamental group of a hyperbolic 3--manifold, finds the rank of $G$.
\end{sat}

However, it is not possible to give a priori bounds on the complexity of
the algorithm provided by the Kapovich--Weidmann theorem and hence it seems
difficult to use it directly to obtain precise results in concrete
situations.

Here, we show how to derive from the results in the previous section the
following theorem analogous to \fullref{genus-mapping1}.

\begin{sat}[Souto \cite{Souto-rank}]\label{rank-mapping1}
Let $\Sigma_g$ be the closed surface of genus $g\ge 2$,
$\phi\in\Map(\Sigma_g)$ a pseudo-Anosov mapping class and $M_{\phi^n}$ the
mapping torus of $\phi^n$. There is $n_\phi$ such that for all $n\ge
n_\phi$
$$\rank(\pi_1(M_{\phi^n}))=2g+1.$$
Moreover for any such $n$ any generating set of $\pi_1(M_{\phi^n})$ with
minimal cardinality is Nielsen equivalent to an standard generating set.
\end{sat}

Observe that by construction we have $\pi_1(M_\phi)=\pi_1(\Sigma_g)*_\BZ$
and hence, considering $2g$ generators of $\pi_1(\Sigma_g)$ and adding a
further element corresponding to the HNN-extension we obtain generating
sets of $\pi_1(M_\phi)$ with $2g+1$ elements. These are the so-called {\em
standard\/} generating sets.

Before going further, we would like to remark that recently an extension of
\fullref{rank-mapping1} has been obtained by Ian Biringer.

\begin{sat}[Biringer]\label{rank-mapping2}
For every $g$ and $\epsilon$ there are at most finitely many hyperbolic
3--manifolds $M$ with $\inj(M)\ge\epsilon$, fibering over the circle with
fiber $\Sigma_g$ and with $\rank(M)\neq 2g+1$.
\end{sat}

\begin{proof} {\rm Sketch proof of \fullref{rank-mapping1}} \qua
Given some $\phi$
pseudo-Anosov, let $M_n=M_{\phi^n}$ be the mapping torus of $\phi^n$. Let
also $f_n \co X_n\to M_n$ be a minimal length carrier graph for each $n$.
Moreover, let $Y_n\subset X_n$ be a sequence of subgraphs of $X_n$ with the
following properties.
\begin{itemize}
\item There is some constant $C$ with $l_{f_n \co X_n\to M_n}(Y_n)\le C$ for
all $n$.
\item If $Y_n'\subset X_n$ is a sequence of subgraphs of $X_n$ properly
containing $Y_n$ for all $n$ then $\lim_nl_{f_n \co X_n\to M_n}(Y_n')=\infty$.
\end{itemize}
We claim that for all sufficiently large $n$ the graph $Y_n$ has a
connected component $\hat Y_n$ such that the image of $\pi_1(\hat Y_n)$ in
$\pi_1(M_n)$ generates the fundamental group of the fiber. In particular we
have
$$2g\le\rank(\pi_1(\hat Y_n))<\rank(\pi_1(M_n))\le 2g+1.$$
The claim of \fullref{rank-mapping1} follows.

Seeking a contradiction assume that the subgraphs $\hat Y_n$ don't exist
for some subsequence $(n_i)$ and let $\wbar Y_{n_i}$ be a sequence of
connected components in $Y_{n_i}$. Since the graph $Y_{n_i}$ has at most
length $C$ we may assume, up to forgetting finitely many, that for all $i$
the graph $\wbar Y_{n_i}$ lifts to the infinite cyclic cover $M'$ of
$M_{n_i}$. Moreover, the space of graphs in $M'$ with bounded length is, up
to the natural $\BZ$ action, compact. In particular, we may assume, up to
taking a further subsequence and perhaps shifting our lift by a
deck-transformation, that the images of $\pi_1(\wbar Y_{n_i})$ and
$\pi_1(\wbar Y_{n_j})$ coincide for all $i$ and $j$.

On the other hand, by assumption $\pi_1(\wbar Y_{n_i})$ does not generate
$\pi_1(M')=\pi_1(\Sigma_g)$. Moreover, since every proper finite index
subgroup of $\pi_1(M')$ has rank larger than $2g$, we have that the image
of $\pi_1(\wbar Y_{n_i})$ in $\pi_1(M')$ has infinite index and hence is a
free group. This implies that the homomorphism $\pi_1(\wbar Y_{n_i})\to
\pi_1(M')$ is injective for otherwise we would be able to find carrier
graphs for $M_n$ with rank less than $\rank(\pi_1(X_n))$ (compare with the
remark following \fullref{Nielsen}). Thurston's \hyperlink{CThe}{Covering Theorem} (compare
with \fullref{ugly}) implies now that in fact the image of $\pi_1(\wbar
Y_{n_i})$ is convex-cocompact. In particular, the quotient under
$\pi_1(\wbar Y_{n_i})$ of its thick-convex-hull has bounded diameter. Then,
minimality of the graph $f_{n_i} \co X_{n_i}\to M_{n_i}$ implies that there is
some $D$ such that for every edge $e$ of $X_{n_i}$ one has
\begin{equation}\label{length-rellength1}
l_{f_{n_i} \co X_{n_i}\to M_{n_i}}(e)\le D+l_{f_{n_i} \co X_{n_i}\to
M_{n_i},\rel(\wbar Y_i)}(e).
\end{equation}
Equation \eqref{length-rellength1} applies by assumption to all components
$\wbar Y_i$ of $Y_i$. In particular we obtain some $D'$ depending on $D$ and
the maximal number of components, ie of $D$ and $g$, such that
\begin{equation}\label{length-rellength2}
l_{f_{n_i} \co X_{n_i}\to M_{n_i}}(e)\le D'+l_{f_{n_i} \co X_{n_i}\to
M_{n_i},\rel(Y_i)}(e).
\end{equation}
Equation \eqref{length-rellength2} contradicts \fullref{bounded-chains}.
This concludes the (sketch of the) proof of \fullref{rank-mapping1}.
\end{proof}

\section{Another nice family of manifolds}\label{sec:hossein}
Until now, we have mostly considered 3--manifolds arising as a mapping
torus. Mapping tori are particularly nice 3--manifolds whose construction is
also particularly simple to describe. But this is not the real reason why
we were until now mainly concerned with them. The underlying geometric
facts needed in the proofs of \fullref{genus-mapping1} and
\fullref{rank-mapping1} are the following:
\begin{itemize}
\item Thurston's theorem asserting that the mapping torus $M_\phi$ of a
pseudo-Anosov mapping class $\phi$ admits a hyperbolic metric and
\item the fact that every geometric limit of the sequence $(M_{\phi^n})_n$
is isometric to the infinite cyclic cover corresponding to the fiber.
\end{itemize}
In some cases, for instance in the proof of \fullref{genus-mapping1}, one
does not need the full understanding of the possible geometric limits. But
still, without understanding enough of the geometry it is not possible,
using our methods, to obtain topological facts. In this section we discuss
results obtained by Hossein Namazi and the author \cite{Namazi-Souto}
concerning a different family of 3--manifolds.

Let $M^+$ and $M^-$ be $3$-dimensional handlebodies of genus $g>1$ with
homeomorphic boundary $\bnd M^+=\bnd M^-$. Given a mapping class
$f\in\Map(\bnd M^+)$ we consider the closed, oriented 3--manifold
$$N_f=M^+\cup_f M^-$$
obtained by identifying the boundaries of $M^+$ and $M^-$ via $f$. By
construction, the manifold $N_f$ has a {\em standard\/} Heegaard splitting of
genus $g$. Also, generating sets of the free groups $\pi_1(M^\pm)$ generate
$\pi_1(N_f)$ and hence $\rank(\pi_1(N_f))\le g$. We will be interested in
those manifolds $N_{f^n}$ where the gluing map is a high power of a
sufficiently complicated mapping class hoping that if this is the case,
then
$$g(N_{f^n})=\rank(\pi_1(N_{f^n}))=g.$$
However, we must be a little bit careful with what we mean under
``complicated mapping class". The problem is that it does not suffice to
assume that $f$ is pseudo-Anosov because there are homeomorphisms $F$ of
$M^+$ which induce pseudo-Anosov mapping classes $f$ on $\bnd M^+$ and for
any such map we have $N_f=N_{f^n}$ for all $n$. For example, $N_{f^n}$
could be the 3--sphere for all $n$. In particular, we have to rule out that
$f$ extends to either $M^+$ or $M^-$. In order to give a precise sufficient
condition recall that every pseudo-Anosov map has an stable lamination
$\lambda^+$ and an unstable lamination $\lambda^-$. If $f$ extends to a
homeomorphism of $M^+$ then it maps meridians, ie essential simple closed
curves on $\bnd M^+$ which are homotopically trivial in $M^+$, to
meridians; in particular, $\lambda^+$ is a limit in the space $\CP\CM\CL$
of measured laminations on $\bnd M^+$ of meridians of $M^+$. With this in
mind, we say that a pseudo-Anosov mapping class $f$ on $\bnd M^+=\bnd M^-$
is {\em generic\/} if the following two conditions hold:
\begin{itemize}
\item the stable lamination is not a limit of meridians in $\bnd M^+$ and
\item the unstable lamination is not a limit of meridians in $\bnd M^-$.
\end{itemize}
The term {\em generic\/} is justified because Kerckhoff \cite{Ker90} proved
that the closure in $\CP\CM\CL$ of the set of meridians of $M^+$ and $M^-$
have zero measure with respect to the canonical measure class of
$\CP\CM\CL$. Moreover, it is not difficult to construct examples of generic
pseudo-Anosov maps by hand.
%

In this section we consider manifolds $N_{f^n}=M^+\cup_{f^n} M^-$ obtained
by gluing $M^+$ and $M^-$ by a high power of a generic pseudo-Anosov
mapping class.

We should point out that the above construction is due to Feng Luo by
using an idea of Kobayashi: the manifolds $N_{f^n}$ are also
interesting because the standard Heegaard splittings can have
arbitrarily large distance in the curve complex. In particular, by
\fullref{Hempel}, the manifold $N_{f^n}$ is irreducible and atoroidal
for all sufficiently large $n$. Hence, it should be hyperbolic. In
fact, according to Perelman's \hyperlink{HypTh}{Hyperbolization Theorem} the manifold
$N_{f^n}$ is hyperbolic provided that its fundamental group is not
finite. To check that this is the case ought to be easy... one
thinks. It is not. The following is the simplest "topological" proof
known to the author that $N_{f^n}$ is not simply connected for large
$n$.

\begin{lem}\label{not-sphere}
For all sufficiently large $n$ the manifold $N_{f^n}$ is not simply
connected.
\end{lem}
\begin{proof}
By the proof of the Poincar\'e conjecture, also by Perelman, it suffices to
prove that $N_{f^n}$ is not the sphere $\BS^3$. However, by Waldhausen's
classification of the Heegaard splittings of the sphere, there is only one
for each $g$. And this one is reducible and hence has distance $0$ in the
curve-complex. The standard Heegaard splitting of $N_{f^n}$ has, for large
$n$, large distance and therefore $N_{f^n}\neq\BS^3$.
\end{proof}

Using the classification of the Heegaard splittings of the lens spaces one
can also check that for large $n$ the manifold $N_{f^n}$ is not a lens
space. Probably something similar can be made to rule out every other
spherical manifold. So, using all these classification theorems for
Heegaard splittings and the geometrization conjecture one finally obtains
that the fundamental group of $N_{f^n}$ is infinite.

The discussion that we just concluded shows how difficult is to prove
anything using topological methods for the manifolds $N_{f^n}$. Even
considering the proof of the geometrization conjucture to be {\rm
topological\/}.

In Namazi--Souto \cite{Namazi-Souto} we don't show that $N_{f^n}$ is hyperbolic but we
construct, by gluing known hyperbolic metrics on $M^+$ and $M^-$, explicit
negatively curve metrics on $N_{f^n}$ for large $n$.

\begin{sat}\label{main}
For arbitrary $\epsilon>0$ there is $n_\epsilon$ such that the manifold
$N_{f^n}$ admits a Riemannian metric $\rho_n$ with all sectional curvatures
pinched by $-1-\epsilon$ and $-1+\epsilon$ for all $n\geq n_\epsilon$.
Moreover, the injectivity radius of the metric $\rho_n$ is bounded from
below independently of $n$ and $\epsilon$.
\end{sat}

The idea behind the proof of \fullref{main} can be summarized as follows:
By the work of Kleineidam and the author \cite{KS02} there are two
hyperbolic manifolds homeomorphic to $M^+$ and $M^-$ which have
respectively ending laminations $\lambda^+$ and $\lambda^-$. The ends of
these two hyperbolic 3--manifolds are asymptotically isometric. In
particular, it is possible to construct $N_{f^n}$, for large $n$, by gluing
compact pieces of both manifolds by maps very close to being an isometry.
On the gluing region, a convex-combinations of both hyperbolic metrics
yields a negatively curved metric.

\begin{figure}[ht!]
\begin{center}
\includegraphics[scale=0.4,angle=270]{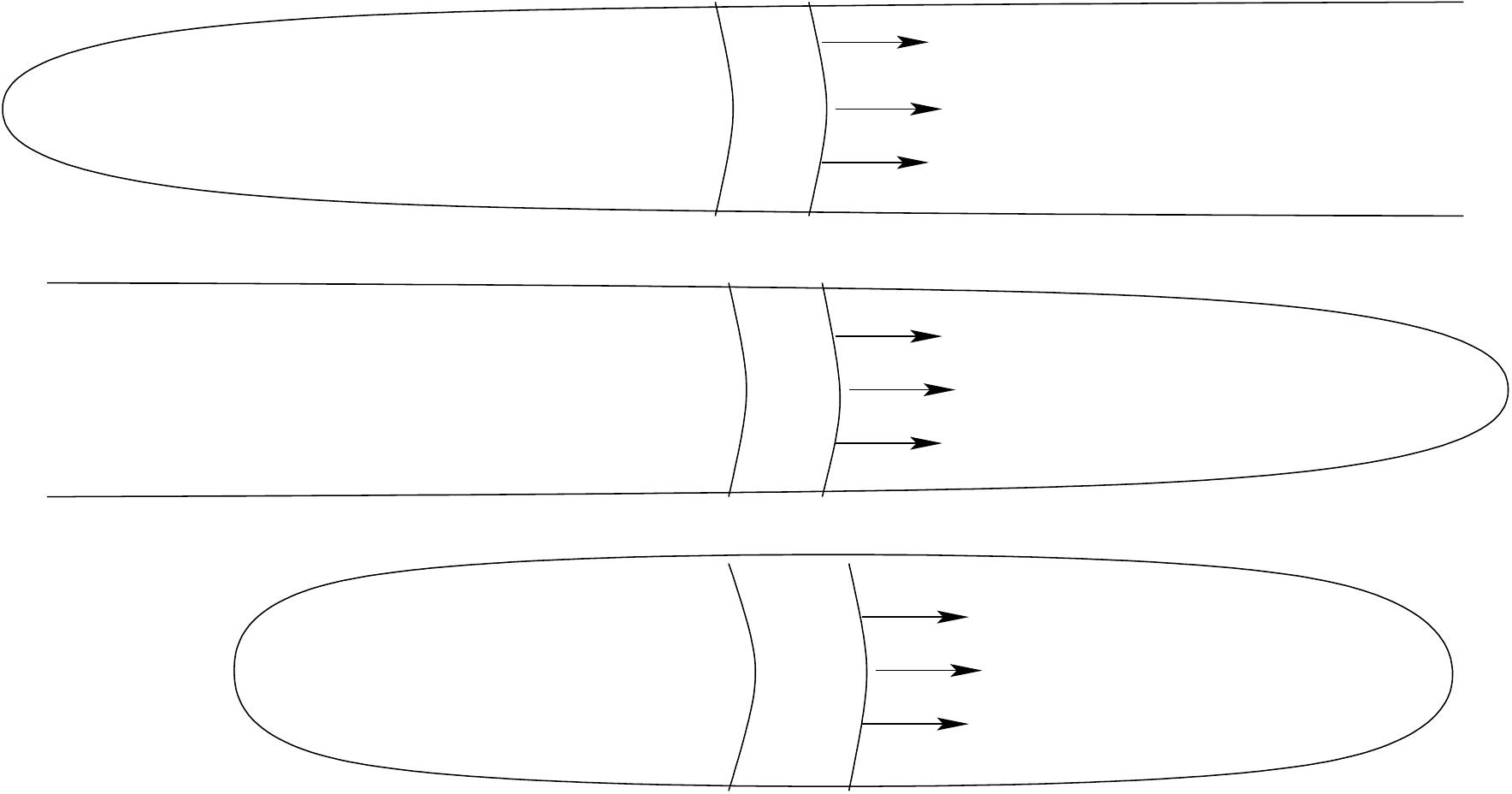}
\end{center}
\caption{Proof of \fullref{main}: The two open manifolds $M^+$ and $M^-$
being glued along the gluing region to obtain the manifold $N_{f^n}$.}
\label{fig:4}
\end{figure}

Apart from providing a negatively curved metric, the key point to all
further results, is that the metrics $\rho_n$ are {\em explicits\/}. In
particular, as one sees from the sketch of the construction above, one has
automatically a complete understanding of the possible geometric limits of
the sequence $(N_{f^n},\rho_n)$. We obtain for example the following theorem.

\begin{sat}\label{geometric limit}
Every geometric limit of the sequence $(N_{f^n})_n$ is hyperbolic and is
either homeomorphic to a handlebody of genus $g$ or to a trivial interval
bundle over a closed surface of genus $g$.
\end{sat}

\fullref{main} shows that $\pi_1(N_{f^n})$ is infinite and word hyperbolic
for all large $n$. In particular, it has solvable word problem and many
other highly desirable algorithmic properties. However, this does not say
much about the specific group $\pi_1(N_{f^n})$; it just says as much as the
geometrization conjecture does. For instance, it does not explain to which
extent the homomorphism $\pi_1(M^+)\to\pi_1(N_{f^n})$ is injective.

\begin{sat}\label{easson}
If $\Gamma\subset\pi_1(M^+)$ is a finitely generated subgroup of infinite
index, then there is some $n_\Gamma$ such that for all $n\geq n_\Gamma$ the
map $\Gamma\to\pi_1(N_{f^n})$ given by the inclusion $M^+\hookrightarrow
N_{f^n}$ is injective.
\end{sat}

The proof of \fullref{easson} is based on the fact that $M^+$ can be
canonically identified with one of the possible geometric limits of the
sequence $(N_{f^n},\rho_n)$ and that, by \fullref{ugly}, every infinite
degree covering of this geometric limit is convex-cocompact; compare with
\fullref{super-ugly}.

Another consequence of \fullref{main} and \fullref{geometric limit} is
that, again for large $n$, the fundamental group of the manifold $N_{f^n}$
has rank $g$. The proof of this fact is almost word-by-word the same as the
proof of \fullref{rank-mapping1}.

\begin{sat}\label{rank-hossein}
There is $n_f$ with $\rank(\pi_1(N_{f^n}))=g$ for all $n\ge n_f$. Moreover,
every minimal generating set of $\pi_1(N_{f^n})$ is Nielsen equivalent to a
standard generating set for all sufficiently large $n$. In particular
$\pi_1(N_{f^n})$ has at most 2 Nielsen equivalence classes of minimal
generating sets.
\end{sat}

In \cite{Namazi-Souto} we also use similar arguments as in the proof of
\fullref{rank-mapping1} and \fullref{rank-hossein} to prove for example
that for sufficiently large $n$, every set of at most $2g-2$ elements in
$\pi_1(N_{f^n})$ which generate a proper subgroup does in fact generate a
free subgroup. This bound is sharp. Finally, we use the same arguments as
outlined in the proof of \fullref{genus-mapping1} to prove that $N_{f^n}$
has, again for large $n$, Heegaard genus $g$ and that the standard Heegaard
splitting is, up to isotopy, the unique minimal genus Heegaard splitting.

\begin{sat}\label{minimal genus}
There is $n_f$ such that for all $n\ge n_f$ the following holds: every
minimal genus Heegaard splitting of $N_{f^n}$ is isotopic to the standard
one.
\end{sat}

Observe that the bound on the Heegaard genus follows also from
\fullref{rank-hossein}. \fullref{minimal genus} is also due to
Scharlemann--Tomova \cite{Scharlemann-Tomova} but their methods are
completely different.

As mentioned above, \fullref{main} and \fullref{geometric limit} are the
key to all the subsequent topological results. While \fullref{main} is a
consequence of the positive answer to Thurston's geometrization conjecture,
the author would like to remark that even assuming the mere existence of
negatively curved metric, or even hyperbolic, on $N_{f^n}$ it is not
obvious how to derive any of the topological applications outlined above.
Recall for example that it was not obvious how to prove, even using the
geometrization conjecture, that $N_{f^n}$ is for large $n$ not finitely
covered by $\BS^3$. It is the control on the geometry of the manifolds
$N_{f^n}$ provided by \fullref{geometric limit} that opens the door to all
the subsequent results.

\section{Bounding the volume in terms of combinatorial
data}\label{sec:volume}
Until now we have mainly studied quite particular classes of manifolds; in
this section our point of view changes. Here we discuss results due to
Brock \cite{Brock-pants}, and Brock and the author, showing that it is
possible to give linear upper and lower bounds for the volume of a
hyperbolic 3--manifold in terms of combinatorial distances.

Given a closed surface $S$ let $\CP(S)$ be its {\em pants-complex\/}, ie
the graph whose vertices are isotopy classes of pants decomposition and
where two pants decompositions $P$ and $P'$ are at distance one if they
differ by an elementary move. Here, two pants decompositions
$P=\{\gamma_1,\dots,\gamma_{3g-3}\}$ and
$P'=\{\gamma_1',\dots,\gamma_{3g-3}'\}$ differ by an elementary move if
there is some $j$ such that
\begin{itemize}
\item $\gamma_i=\gamma_i'$ for all $i\neq j$ and
\item such that $\gamma_j$ and $\gamma_j'$ are different curves with the
minimal possible intersection number in
$S\setminus\{\gamma_1,\dots,\gamma_{j-1},\gamma_{j+1},\dots,\gamma_{3g-3}\}$.
\end{itemize}
In more mundane terms, in order to change $P$ by an elementary move, one
keeps all components of $P$ but one fixed and this component is changed in
the simplest possible way.

The pants complex $\CP(S)$ is known to be connected. In particular,
declaring every edge to have unit length one obtains an interior distance
invariant under the natural action of the mapping class group on the pants
complex. From our point of view, the pants complex is important because of
the following result due to Brock \cite{Brock-pants} relating volumes of
mapping tori to distances in the pants complex.

\begin{sat}[Brock \cite{Brock-pants}]\label{volume1}
For every $g$ there is a constant $L_g>0$ with
$$L_g^{-1}l_{\CP(\Sigma_g)}(f)\le\vol(M_f)\le L_gl_{\CP(\Sigma_g)}(f)$$
for every pseudo-Anosov mapping class $f$ of the surface $\Sigma_g$ of
genus $g$. Here $M_f$ is the mapping torus of $f$ and
$$l_{\CP(\Sigma_g)}(f)=\inf\{d_{\CP(\Sigma_g)}(P,f(P))\vert
P\in\CP(\Sigma_g)\}$$
is the translation length of $f$ in the pants complex of $\Sigma_g$.
\end{sat}

It is a beautiful observation due to Brock \cite{Brock-pants} that the
pants complex is quasi-isometric to the Teichm\"uller metric when endowed
with the Weil-Petersson metric. In particular, \fullref{volume1} is stated
in \cite{Brock-pants} in terms of translation distances in the
Teichm\"uller space.

In the setting of \fullref{volume1}, getting upper bounds for the volume is
not difficult. The idea is that if $(P_1,P_2,\dots,P_d)$ is a path in the
pants complex with $P_d=f(P_1)$ then one obtains an ideal triangulation of
the mapping torus $M_f$ with some controlled number of simplices. In
particular, the translation length bounds the simplicial volume of $M_f$,
and hence bounds the volume itself. Agol \cite[Cor 2.4]{Agol-small}
obtained sharp upper bounds for the volume of $M_f$ in terms of
$l_{\CP(\Sigma_g)}(f)$ from which it follows for example that
$$\vol(M_f)\le 2V_{oct}l_{\CP(\Sigma_g)}(f)$$
where $V_{oct}$ is the volume of a regular ideal octahedron.

\begin{proof}
The bulk of the proof of \fullref{volume1} is to show that the volume can
be bounded from below in terms of the translation distance in the
pants-complex. In \cite{Brock-pants}, this lower bound is derived using
deep difficult results due to Masur--Minsky
\cite{Masur-Minsky1,Masur-Minsky2} about the geometry of the curve-complex
and exploiting the relation between the curve-complex and the geometry of
hyperbolic manifolds discovered by Minsky. However, a simpler proof of the
lower bound was found by Brock and the author.

The first observation is that the volume of $M_f$ is roughly the same as
the volume of its thick part. The idea of the proof is to decompose the
thick part of $M_f$ in a collection of disjoint pieces, called pockets, and
then use geometric limit arguments to estimate the volume of each single
pocket.

Consider for the time being only mapping tori $M_f$ with injectivity radius
at least $\epsilon$. This assumption has the following consequence:
\begin{itemize}
\item[(*)] If $p_i$ is a sequence of points in different mapping tori
$M_{f_i}$ with $\inj(M_{f_i})\ge\epsilon$ for all $i$ then, up to passing
to a subsequence, the sequence of pointed manifolds $(M_{f_i},p_i)$
converges geometrically to a pointed hyperbolic manifold
$(M_\infty,p_\infty)$ homeomorphic to $\Sigma_g\times\BR$ and with at least
injectivity radius $\epsilon$.
\end{itemize}
In some way (*) gives full geometric control. For instance, it is known
that every point $p$ in a mapping torus $M_f$ is contained in a surface
$S_p$ homotopic to the fiber of $M_f$ and with curvature $\le -1$.
Thurston's bounded diameter lemma asserts that
$$\diam(S_p)\le\frac{4g-4}{\epsilon^2}$$
and hence it follows directly from (*) that there is a constant $c_1$
depending of $\epsilon$ and $g$ such that if $M_f$ has at least injectivity
radius $\epsilon$ then for all $p\in M_f$ the surface $S_p$ is in fact
homotopic to an embedded surface $S_p'$ by a homotopy whose tracks have at
most length $c_1$ (see footnote\footnote{In order to see that the constant
$c_1$ exists we can proceed as follows: given a surface $S$ in a 3--manifold
consider the infimum of all possible lengths of tracks of homotopies
between $S$ and an embedded surface; this infimum may be $\infty$ if such a
homotopy does not exist. Then it is easy to see that this function is semi
continuous in the geometric topology and hence bounded on compact sets; (*)
is a compactness statement.}). For every $p$ we choose a minimal length
pants decomposition $P_p$ in the surface $S_p$.

Choose $D$ to be a sufficiently large constant, depending on $g$ and
$\epsilon$. Given a mapping torus $M_f$ with injectivity radius
$\inj(M_f)\ge\epsilon$ let $\{p_1,\dots,p_k\}$ be a maximal collection of
points in $M_f$ which are at at least distance $D$ from each other. Observe
that for every $i$ there is $j\neq i$ such that $p_i$ and $p_j$ are at at
most distance $2D$. For each $i=1,\dots,k$ we consider the negatively
curved surface $S_{p_i}$ passing trough $p_i$ and the associated embedded
surface $S_{p_i}'$. Then, since $D$ is large and the diameter of the
surfaces $S_{p_i}$ is bounded we have that the surfaces $S_{p_i}$ and
$S_{p_j}$ are disjoint for all $i\neq j$; the same also holds for the
associated embedded surfaces $S_{p_i}'$ and $S_{p_j}'$ by the bound on the
length of the tracks of the homotopies. In particular, the surfaces
$S_{p_i}'$ and $S_{p_j}'$ are disjoint, embedded and homotopic to the
fiber. We obtain from Waldhausen's cobordism theorem that they bound in
$M_f$ a product region homeomorphic to $\Sigma_g\times[0,1]$. In
particular, the set $\{p_1,\dots,p_k\}$ is cyclically ordered, say as
$[p_1,\dots,p_k]$.

\begin{figure}[ht!]
\begin{center}
\includegraphics[scale=0.6,angle=270]{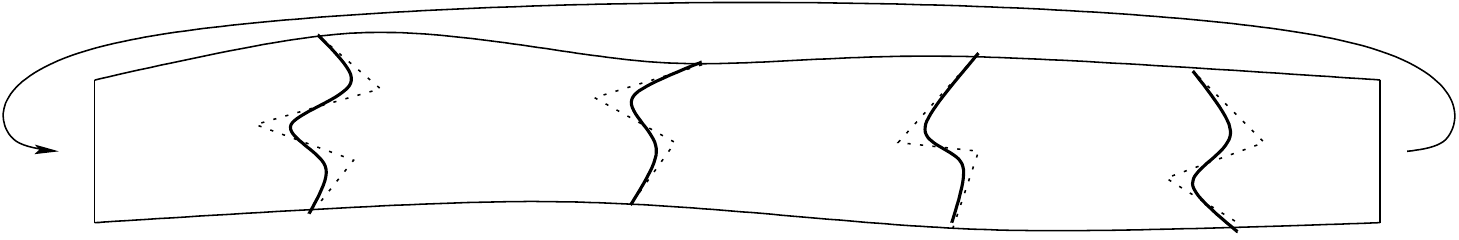}
\end{center}
\caption{Proof of \fullref{volume1} in the thick case: The dotted lines are
simplicial hyperbolic surfaces and the thick lines are close-by embedded
surfaces dividing the manifold $M_f$ into roughly uniform pieces.}
\label{fig:5}
\end{figure}

In other words, under the assumption that $\inj(M_f)\ge\epsilon$ and that
$\diam(M_g)\ge 10D$, we have that a decomposition of the mapping torus
$M_f$ into cyclically ordered product regions $[U_1,\dots,U_k]$ with
$$\bnd U_1=S_{p_1}'\cup S_{p_2}',\cdots,\bnd U_{k-1}=S_{p_{k-1}}'\cup
S_{p_k}',\bnd U_k=S_{p_k}'\cup S_{p_1}'$$
and such that for all $i=1,\dots,k$ if $p,q\in\bnd U_i$ are in different
components then
$$D-2\frac{4g-4}{\epsilon^2}\le d_{M_f}(p,q)\le
2D+2\frac{4g-4}{\epsilon^2}.$$
These bounds, together with (*) implies that there are constants $V$ and
$L$ depending only on $g$ and $\epsilon$ such that for every one of the
product regions $U_i$ we have
$$\vol(U_i)\ge V, \qua d_{\CP(\Sigma_g)}(P_{p_i},P_{p_{i+1}})\le L.$$
From the first inequality we obtain that there are at most $\frac
1V\vol(M_f)$ product regions. From the second we deduce that
$$d_{\CP(\Sigma_g)}(P_{p_1},f(P_{p_1}))\le\frac LV\vol(M_f).$$
This concludes the proof of \fullref{volume1} under the additional
assumption that the mapping torus in question has injectivity radius
$\inj(M_f)\ge\epsilon$.

In general the idea is to decompose the thick part of $M_f$ into product
regions. The key point is the following well-known result due to Otal
\cite{Otal-unknot}.

\begin{sat}[Otal's unknotting theorem]
For every $g$ there is a constant $\epsilon_g$ such that for every
$f \co \Sigma_g\to\Sigma_g$ pseudo-Anosov the following holds: There is a
collection of disjoint surfaces $S_1,\dots,S_k$ parallel to
$\Sigma_g\times\{0\}$ with the property that every primitive geodesic in
$M_f$ shorter than $\epsilon_g$ is contained in $\cup S_i$.
\end{sat}

In the sequel denote by $\Gamma$ the collection of primitive geodesics in
$M_f$ that are shorter than $\epsilon_g$. By Thurston's
\hyperlink{HypTh}{Hyperbolization Theorem}
the manifold $M_f^*=M_f\setminus\Gamma$ admits a complete finite
volume hyperbolic metric. Moreover, it follows from the deformation theory
of cone-manifolds due to Hodgson--Kerckhoff \cite{Hodgson-Kerckhoff} that up
to assuming that $\epsilon_g$ is smaller than some other universal constant
the ratio between the volumes of $M_f$ and $M_f^*$ is close to 1, and
\begin{itemize}
\item[(**)] every geodesic in $M_f^*$ has at least length
$\frac{\epsilon_g}2$. \end{itemize}

\begin{figure}[ht!]
\begin{center}
\includegraphics[scale=0.4,angle=270]{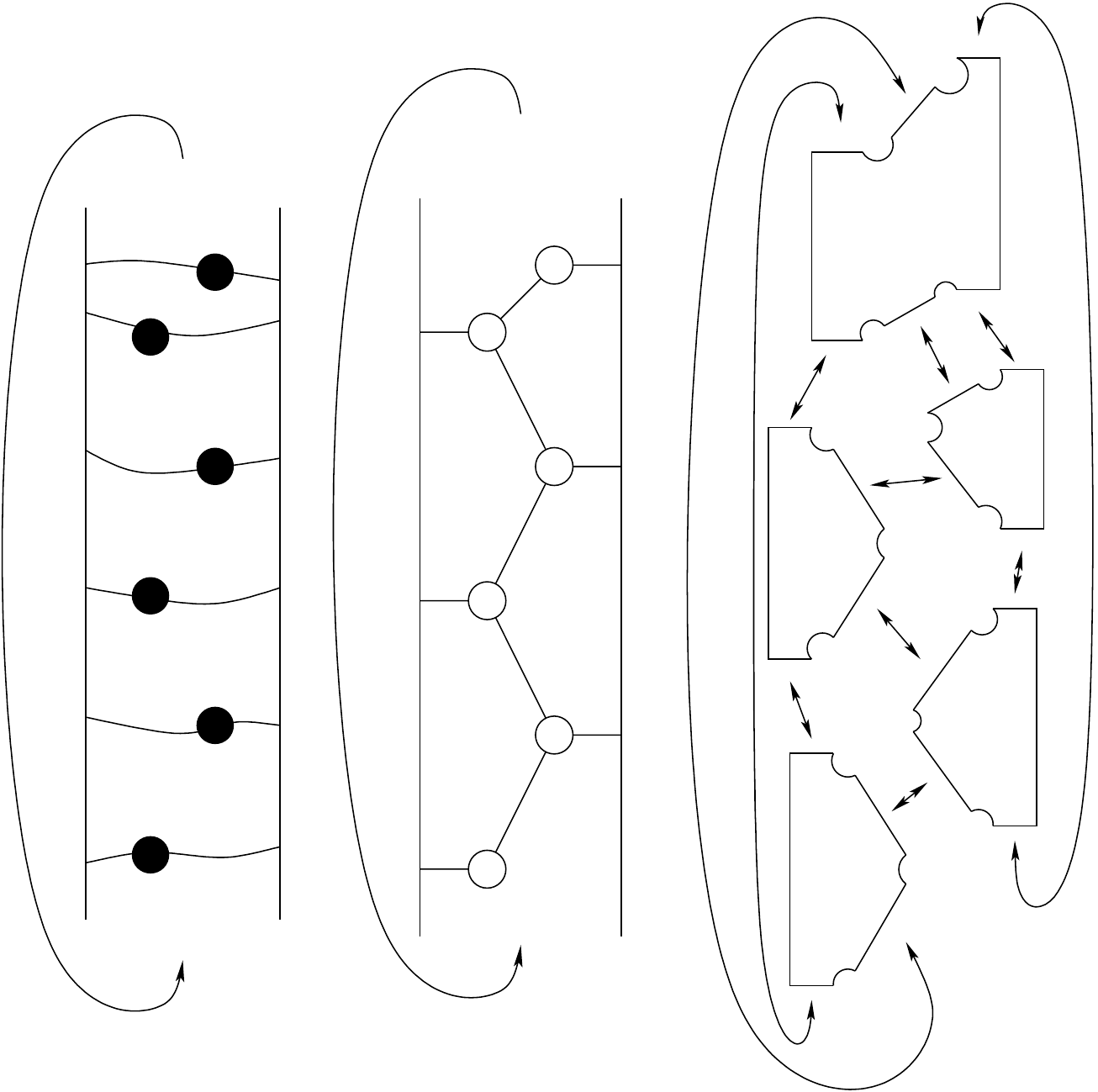}
\end{center}
\caption{Proof of \fullref{volume1}: The first picture represents the
manifold $M_f$ where the black dots are short geodesics and the lines are
the surfaces provided by Otal's unknotting theorem. In the second picture
one has the manifold $M_f^*$ and the straight lines are obtained from
Otal's surfaces by (1) isotoping as much as possible into the cusps, (2)
removing parallel components, and (3) pulling tight. In the third image one
sees the associated pocket decomposition.}
\label{fig:6}
\end{figure}

\fullref{thm:incompressible-minimal} and Otal's theorem imply that $M_f^*$
can be cut open long minimal surfaces into disjoint product regions
$V_i=F_i\times(0,1)$ where $F_i$ is a subsurface of $\Sigma_g$; the $V_i$'s
are the so-called pockets. The volume of $M_f^*$ is the sum of the volumes
of the regions $V_i$. Moreover, choosing for each $i$ minimal length pants
decompositions $P_i^-$ and $P_i^+$ of the boundary of $V_i$ it is not
difficult to see that
$$l_{\CP(\Sigma_g)}\le\sum_{V_i} d_{\CP(F_i)}(P_i^-,P_i^+).$$
In particular, it suffices to prove that the volume of each one of the
product regions $V_i$ bounds the distance $d_{\CP(F_i)}(P_i^-,P_i^+)$ from
above. However, since in the setting we have (**) we are again in the
situation that $V_i$ is thick: the same arguments used above yield the
desired result in this case.

This concludes the sketch of the proof of \fullref{volume1}.
\end{proof}

The following is an analogous of Otal's theorem in the setting of Heegaard
splittings.

\begin{sat}[Souto]\label{unknotting}
For every $g$ there is a constant $\epsilon_g$ such that for every strongly
irreducible genus $g$ Heegaard surface $S$ in a closed hyperbolic
3--manifold $M$ the following holds: There is a collection of disjoint
surfaces $S_1,\dots,S_k$ parallel to $S$ with the property that every
primitive geodesic shorter than $\epsilon_g$ in $M$ is contained in $\cup
S_i$.
\end{sat}

\begin{proof}
The idea of the proof of \fullref{unknotting} is the following. By
Pitts--Rubinstein's \fullref{Pitts-Rubinstein}, we can associate to the
surface $S$ a minimal surface $F$. For the sake of simplicity assume that
$F$ is isotopic to $S$ and contained in the thick part of $M$. Cutting $M$
along $S$ we obtain two handlebodies; let $U$ be the metric completion of
one of these handlebodies and $\tilde U$ its universal cover. Then, by a
theorem of Alexander--Berg--Bishop \cite{ABB}, $\tilde U$ is a
CAT$(-1)$-space and hence, from a synthetic point of view, it behaves very
much like hyperbolic space. In particular, one can use word-by-word the
proof in (Souto \cite{Souto-unknotting}) that short curves in hyperbolic handlebodies
are parallel to the boundary and obtain that short curves in the handlebody
$U$ are parallel to the boundary $\bnd U=F\simeq S$.
\end{proof}

\fullref{unknotting} opens the door to analogous results to
\fullref{volume1} for Heegaard splittings. In some sense, the main
difficulty is to decide what has to replace the translation length
$l_\CP(f)$.

\begin{defi*}
The {\em handlebody set\/} $\CH(H)$ of a handlebody $H$ is the subset of the
pants complex $\CP(\bnd H)$ of its boundary consisting of those pants
decompositions $P$ with the following property: There is a collection
$\calD$ of properly embedded disks in $H$ with boundary in $P$ and such
that $H\setminus\calD$ is homeomorphic to a collection of solid tori.
\end{defi*}

Given a Heegaard splitting $M=U\cup V$ of a 3--manifold we can now define
its distance in the pants complex of the Heegaard surface $S$ as follows:
$$\delta_\CP(U,V)=\min\{d_{\CP(S)}(P_U,P_V)\vert
P_U\in\CH(U),P_V\in\CH(P_V)\}.$$
\fullref{unknotting} and similar geometric limit arguments as the ones used
in the proof of \fullref{volume1} yield the following theorem.

\begin{sat}[Brock-Souto]\label{volume2}
For every $g$ there is a constant $L_g>0$ with
$$L_g^{-1}\delta_\CP(U,V)\le\vol(M)\le L_g\delta_\CP(U,V)$$
for every genus $g$ strongly irreducible Heegaard spitting $M=U\cup V$ of a
hyperbolic 3--manifold $M$.
\end{sat}

\section{Hyperbolic manifolds with given rank}\label{sec:rank2}
In the last section we studied the relation between the geometry of
hyperbolic 3--manifolds and the combinatorics of Heegaard splittings. In
this section we sketch some results about what can be said about the
geometry of a hyperbolic 3--manifold with given rank of the fundamental
group. We will be mainly interested in the following two conjectures which
assert that (A) the radius of the largest embedded ball in a closed
hyperbolic 3--manifold and (B) its Heegaard genus are bounded from above in
terms of the rank of the fundamental group.

\begin{conj}[McMullen]\label{conjecturea}
For all $k$ there is some $R$ with
$$\inj(M,x)\le R$$
for all $x$ in a closed hyperbolic 3--manifold $M$ with $\rank(\pi_1(M))=k$.
\end{conj}

Observe that $\max_{x\in M}\inj(M,x)$ is the radius of the largest embedded
ball in $M$. McMullen's conjecture admits a suitable generalization to the
setting of infinite volume hyperbolic 3--manifolds. In fact, it was in this
setting in which McMullen's conjecture was first formulated because of its
implications to holomorphic dynamics. However, it follows from the work of
Gero Kleineidam and the author of this note that the general case can be
reduced to the closed case. Before going further we should recall the
injectivity radius itself, ie the minimum of the injectivity radius over
all points of the manifold, is bounded from above in terms of the rank by
White's \fullref{white-inj}.

\begin{conj}[Waldhausen]\label{conjectureb}
For every $k$ there is $g$ such that every closed 3--manifold with
$\rank(\pi_1(M))=k$ has Heegaard genus $g(M)\le g$.
\end{conj}

Waldhausen asked in fact if for every 3--manifold $\rank(\pi_1(M))=g(M)$.
However, the stronger form of the \fullref{conjectureb} was answered in the
negative by Boileau--Zieschang \cite{Boileau-Zieschang} who presented an
example of a Seifert-fibered 3--manifold $M$ with $g(M)=3$ and
$\rank(\pi_1(M))=2$. Examples of 3--manifolds with $g(M)=4k$ and
$\rank(\pi_1(M))=3k$ were constructed by Schultens-Weidmann
\cite{Schultens-Weidmann}. Recently Abert--Nikolov \cite{Abert-Nikolov} have
announced that there are also hyperbolic 3--manifolds with larger Heegaard
genus than rank.

\fullref{conjecturea} and \fullref{conjectureb}
are related by a result of Bachmann--Cooper--White
\cite{BCW} who proved the following theorem.

\begin{sat}[Bachmann--Cooper--White]\label{BCW}
Suppose that $M$ is a closed, orientable, connected Riemannian $3$--manifold
with all sectional curvatures less than or equal to $-1$ and with Heegaard
genus $g(M)$. Then
$$g(M)\ge\frac{\cosh(r)+1}2$$
where $r=\max_{x\in M}\inj_x(M)$.
\end{sat}

The first positive result towards \fullref{conjecturea}
and \fullref{conjectureb} is due to Agol,
who proved the next theorem.

\begin{sat}[Agol]\label{Agol-genus2}
For every $\epsilon>0$ there is $V>0$ such that every hyperbolic 3--manifold
with injectivity radius $\inj(M)>\epsilon$, volume $\vol(M)>V$ and
$\rank(\pi_1(M))=2$ has Heegaard genus $g(M)=2$. In particular, the
injectivity radius at every point is bounded from above by $\arccosh(3)-1$.
\end{sat}

\begin{proof}
Agol's theorem is unfortunately not available in print. However, the idea
of the proof is not so difficult to explain. Assume that $(M_i)$ is a
sequence of hyperbolic 3--manifolds with $\inj(M_i)\ge\epsilon$,
$\rank(\pi_1(M_i))=2$ and $\vol(M_i)\to\infty$. A special case of
\fullref{bounded-chains} implies that there is for all $i$ a minimal length
carrier graph $f_i \co X_i\to M_i$ with length $l_{f_i \co X_i\to M_i}(X_i)$
bounded from above by some universal constant. Passing to a subsequence we
may assume that the manifold $M_i$ converge geometrically to a hyperbolic
manifold $M_\infty$, the graphs $X_i$ to a graph $X_\infty$ and the maps
$f_i$ to a map $f_\infty \co X_\infty\to M_\infty$.

We claim that $M_\infty$ is homeomorphic to a handlebody of genus $2$. In
fact, in order to see that this is the case, it suffices to observe that it
has infinite volume that that its fundamental group is generated by 2
elements. In other words, it suffices to prove that $f_\infty \co X_\infty\to
M_\infty$ is a carrier graph. In order to do so, we consider the covering
of $M_\infty'\to M_\infty$ determined by the image of $\pi_1(X_\infty)$.
The proof of the tameness conjecture by Agol \cite{Agol} and Calegari--Gabai
\cite{Calegari-Gabai} implies that $M_\infty'$ is homeomorphic to a
handlebody. Moreover, $M_\infty'$ is not convex-cocompact because otherwise
the homomorphism $(f_i)_* \co \pi_1(X_i)\to\pi_1(M_i)$ would be injective, and
hence $\pi_1(M_i)$ free, for large $i$. Since $M_\infty'$ is not
convex-cocompact, Thurston and Canary's \hyperlink{CThe}{Covering Theorem}
\cite{Canary-covering} implies that the covering $M_\infty'\to M_\infty$ is
finite-to-one; in fact this covering is trivial because $M_\infty'$ is a
handlebody of genus $2$ and a surface of genus $2$ does not cover any other
surface. This proves that $M_\infty$ itself is a handlebody.

Choose now $C\subset M_\infty$ a compact core, ie a compact submanifold
of $M_\infty$ such that $M_\infty\setminus C$ is homeomorphic to a product.
Pushing back the core $C$ to the approximating manifolds $M_i$ we obtain in
each $M_i$ a handlebody $C_i$. In order to conclude the proof of
\fullref{Agol-genus2} it suffices to show that its complement is a
handlebody as well. Furthermore, it suffices to show that $\bnd C_i$ is
compressible in $M_i\setminus C_i$. However, if this is not the case it is
possible to deduce from the \hyperlink{CThe}{Covering Theorem} that there is a non-trivial
homotopy from $\bnd C_i$ to itself supported in $M_i\setminus C_i$. In
particular, $M_i\setminus C_i$ is homeomorphic to a twisted interval bundle
and hence there is a non-trivial homology class supported in the complement
of $C_i$. This contradicts the assumption that $\pi_1(X_i)$, and hence
$\pi_1(C_i)$, surjects onto $\pi_1(M_i)$. Hence $\bnd C_i$ is, for large
$i$, compressible in $M_i\setminus C_i$ and the latter is a handlebody.
This concludes the sketch of the proof of \fullref{Agol-genus2}.
\end{proof}

As we saw during the sketch of the proof the case of $\rank(\pi_1(M))=2$ is
quite particular because of two reasons.
\begin{itemize}
\item If $M$ is a thick hyperbolic 3--manifold with $\rank(\pi_1(M))=2$ then
there is a carrier graph whose length is uniformly bounded from above.
\item If $M$ is a thick non-compact complete hyperbolic 3--manifold whose
fundamental group is generated by two elements then $M$ is a handlebody.
\end{itemize}
This two facts were heavily used in the proof of \fullref{Agol-genus2}. If
the rank of $\pi_1(M)$ is higher than two, then both statements fail,
However, using a similar strategy as in the proof of \fullref{Agol-genus2},
together with the facts about carrier graphs explained in
\fullref{sec:carrier} it is possible to prove the following claim: Given a
sequence $(M_i)$ of closed hyperbolic 3--manifolds with $\inj(M_i)>\epsilon$
and $\rank(\pi_1(M))=3$ there is a compact, atoroidal and irreducible
3--manifold $N$ and a subsequence $(M_{i_j})$ such that for all $j$ the
manifold $M_{i_j}$ contains a compact submanifold homeomorphic to $N$ whose
complement is a union of handlebodies. In particular, one obtains the
following structure theorem of those manifolds whose fundamental group has
rank 3.

\begin{sat}[Souto \cite{rankII}]\label{rank3-structure}
For all positive $\epsilon$ there is a finite collection $N_1,\dots,N_k$ of
compact, atoroidal and irreducible 3--manifolds such that every closed
hyperbolic 3--manifold $M$ with $\inj(M)>\epsilon$ and $\rank(\pi_1(M))=3$
contains a compact submanifold $N$ homeomorphic $N_i$ for some $i$ such
that $M\setminus N$ is a union of handlebodies.
\end{sat}

From \fullref{rank3-structure} the next theorem follows.

\begin{sat}\label{rank3-genus}
For every $\epsilon>0$ there is some $g$ such that every closed hyperbolic
3--manifold $M$ with $\inj(M)>\epsilon$ and $\rank(\pi_1(M))= 3$ has
Heegaard genus $g(M)\le g$.
\end{sat}

\begin{proof}
Given $\epsilon$ positive let $N_1,\dots,N_k$ be the finite collection of
manifolds provided by \fullref{rank3-structure} and let
$$g=\max\{g(N_1),\dots,g(N_k)\}.$$
If $M$ is a closed hyperbolic 3--manifold with $\inj(M)\ge\epsilon$ and
$\rank(\pi_1(M))=3$ then, the \fullref{rank3-structure}, the manifold $M$
contains a submanifold $N$ homeomorphic to some $N_i$ such that $M\setminus
N$ is a collection of handlebodies. In particular, every Heegaard splitting
of $N$ extends to a Heegaard splitting of $M$; hence $g(M)\le
g(N)=g(N_i)\le g$.
\end{proof}

The next theorem follows from \fullref{BCW}.

\begin{sat}
For all $\epsilon$ positive there is $R$ such that for every closed
hyperbolic $3$-manifold $M$ with $\inj(M)\ge\epsilon$ and
$\rank(\pi_1(M))=3$ and for every $x\in M$ one has $\inj(M,x)\le R$.
\end{sat}

In some sense the claim of the \fullref{rank3-structure} may seem redundant
once one has \fullref{rank3-genus}. However one gets from the proof of
\fullref{rank3-structure} some additional information about the geometry of
the manifolds. For example, one can use this additional structure together
with Pitts--Rubinstein's \fullref{Pitts-Rubinstein} to prove the next theorem.

\begin{sat}
For all $\epsilon$ and $g$ there is a number $k$ such that every
$\epsilon$-thick hyperbolic 3--manifold $M$ with $\rank(\pi_1(M))=3$ admits
at most $k$ isotopy classes of Heegaard surfaces of genus $g$.
\end{sat}

\begin{sat}
For all $\epsilon$ there is $d$ such that every closed $\epsilon$--thick
hyperbolic $3$-manifold $M$ which admits a genus $4$ Heegaard spitting with
at least distance $d$ in the curve complex has $\rank(\pi_1(M))=4$.
\end{sat}

\section{Open questions}\label{sec:questions}

In this section we outline some open questions and problems.

\begin{prob}\label{prob:heeg-model}
Construct models for hyperbolic 3--manifolds in terms of combinatorial data
given by a Heegaard splitting.
\end{prob}

A satisfactory answer of \fullref{prob:heeg-model} would be given by a
machine which, when fed with the combinatorial data of a Heegaard splitting
of genus $g$ of a 3--manifold $M$, yields a metric $\rho$ on $M$ such that
whenever $M$ admits a hyperbolic metric $\rho_0$, then $(M,\rho)$ and
$(M,\rho_0)$ are $L_g$-bi-Lipschitz where $L_g$ depends only on $g$. If
such a machine exists, then the obtained metric $\rho$ captures all the
coarse information about $M$.

Partial results have been obtained towards an answer of
\fullref{prob:heeg-model}. Let $(M,\rho_0)$ be a hyperbolic manifold
and $M=U\cup V$ a Heegaard splitting of genus $g$. Then it is possible to
construct out of combinatorial data determined by the Heegaard splitting
$M=U\cup V$ a metric $\rho$ on $M$ such that $(M,\rho)$ and $(M,\rho_0)$
are $L_{g,\inj(M,\rho_0)}$-bi-Lipschitz. In other words, the constant in
question depends on the genus of the splitting and the injectivitiy radius
of the hyperbolic metric. On the other hand, it is also possible to give
lower bounds on the injectivity radius of the hyperbolic metric in terms of
combinatorial data of the Heegaard splitting of $M$. If $\gamma$ is a
simple closed curve on the Heegaard surface such that $(U,\gamma)$ and
$(V,\gamma)$ has incompressible and acylindrical pared boundary, then it is
also possible to determine if the homotopy class of $\gamma$ is non-trivial
and short in $(M,\rho_0)$. All these results, due to Jeff Brock, Yair
Minsky, Hossein Namazi and the author, are steps towards a positive answer
to \fullref{prob:heeg-model}.

Answering \fullref{prob:heeg-model} also opens the door to obtain
partial proofs of the geometrization conjecture. The idea being that if
under the assumption of the existence of a hyperbolic metric one is able to
construct metrics close to the hyperbolic metric, then one can perhaps be
smarter and use the machine developed to answer
\fullref{prob:heeg-model} to construct hyperbolic, or at least
negatively curved, metrics. In some sense, this is the idea behind the
results presented in \fullref{sec:hossein}; compare with \fullref{main}.
The construction used in the proof of \fullref{sec:hossein} is a special
case of a more general construction due to Hossein Namazi \cite{Namazi}.
Recently, Jeff Brock, Yair Minsky, Hossein Namazi and the author have
proved that for every $g$ there is a constant $D_g$ such that every closed
3--manifold which admits a genus $g$ Heegaard splitting with at least
distance $D_g$ in the curve complex admits also a negatively curved metric.

\begin{prob}\label{prob:weak-geom}
Show that for every $g$ there are at most finitely many counter-examples to
the geometrization conjecture which admits a genus $g$ Heegaard splitting.
\end{prob}

After the proof of the geometrization conjecture by Perelman,
\fullref{prob:weak-geom} may seem redundant. However, a satisfactory
answer to \fullref{prob:weak-geom} would consist of presenting a
combinatorial, Ricci-flow-free, construction of the hyperbolic metric in
question. In fact, after answering \fullref{prob:weak-geom} one could try
to prove that for example there are no counter-examples to the
geometrization conjecture which admit a genus $10$ Heegaard splitting.
Apart from the difficulty of checking if a manifold is hyperbolic or not
there is the conceptual problem that unless the constants involved in the
answers of \fullref{prob:heeg-model} and \fullref{prob:weak-geom}
are computable then there can be no a priori bounds on the number of
possible exceptions to the geometrization conjecture. So far, all similar
constants are obtained using geometric limits and compactness and hence
they are not computable.

\begin{prob}
Obtain explicit constants.
\end{prob}

We turn now to questions related to the rank of the fundamental group. If
$\Sigma_g$ is a closed surface of genus $g$ then it is well-known that
$\rank(\pi_1(\Sigma_g))=2g$. In \cite{Zieschang}, Zieschang proved that in
fact $\pi_1(\Sigma_g)$ has a single Nielsen equivalence class of minimal
generating sets. Zieschang's proof is quite combinatorial and difficult to
read.

\begin{prob}\label{prop:Zieschang}
Give a geometric proof of Zieschang's result.
\end{prob}

A geometric proof of Zieschang's result would also shed some light on the
fact that there are 2--dimensional hyperbolic orbifolds whose fundamental
groups admit two different Nielsen equivalence classes of minimal
generating sets. Compare with Weidmann \cite{Weidmann}.

In general, if $G$ is a finitely generated group then every generating set
$(g_1,\dots,g_r)$ with $r$ elements can be {\em stabilized\/} to a generating
set $(g_1,\dots,g_r,e_G)$ with $r+1$ elements by adding the identity to it.
It is easy to see that any two generating sets with $r$ elements are
Nielsen-equivalent after $r$ stabilizations. Clearly, if two generating
sets with the same cardinality are not Nielsen equivalent, then one needs
at least a stabilization so that they can be connected by a sequence of
Nielsen moves. The author suspects that the two standard generating sets of
the manifolds considered in \fullref{sec:hossein} are not Nielsen
equivalent, and in fact that one need $g$ stabilizations to make them
Nielsen equivalent; here $g$ is the genus of the involved handlebodies.

\begin{prob}\label{problem-rank-mountain}
If $N_{f^n}$ is one of the manifolds considered in \fullref{sec:hossein}
and $n$ is very large, determine how many stabilizations does one need so
that the standard Nielsen equivalence classes of generating sets coincide.
\end{prob}

\begin{bem}
As remarked by Richard Weidmann, \fullref{problem-rank-mountain} can be
easily solved if the genus is equal to 2. His argument is very special to
this situation.
\end{bem}

The relation between Heegaard splittings and geometry is much better
understood than the relation between rank of the fundamental group and
geometry. In some sense, one can take any result relating geometry of
hyperbolic manifolds and Heegaard splittings and try to prove the analogous
statement for the rank of the fundamental group. For example the following
question.

\begin{prob}\label{prob:rank-mapping}
Prove that for every $g$ there is some $D_g$ such that whenever
$\phi\in\Map(\Sigma_g)$ is a pseudo-Anosov mapping class with at least
translation length $D_g$ in the curve complex, then the fundamental group
of the mapping torus $M_\phi$ has rank $2g+1$ and  a single Nielsen
equivalence class of generating sets.
\end{prob}

More ambitiously, one could remark that by \hyperlink{MRTh}{Mostow's Rigidity Theorem} a
hyperbolic 3--manifold is determined by its fundamental group. In
particular, the following questions make sense.

\begin{prob}
Given a presentation of the fundamental group of a 3--manifold estimate the
volume.
\end{prob}

\begin{prob}
Construct models out of the algebraic data provided by a presentation.
\end{prob}

Most probably, satisfactory answers to these questions would imply answers
to \fullref{conjecturea} and \fullref{conjectureb} from \fullref{sec:rank2}.

Another question related to the relation between Heegaard genus and rank of
the fundamental group is the following. As mentioned above, Abert and
Nikolov \cite{Abert-Nikolov} have announced that there are hyperbolic
3--manifolds with larger Heegaard genus than rank. However their proof is
not constructive. Essentially they show that there is no equality for all
but finitely many elements in a certain sequence of hyperbolic 3--manifolds.
It would be interesting to have concrete examples.

\begin{prob}
Construct concrete examples of hyperbolic 3--manifolds with larger rank than
Heegaard genus.
\end{prob}

A possible other interesting line of research would be to use minimal
surfaces to obtain geometric proofs about Heegaard splittings of
non-hyperbolic 3--manifolds. For example, using that the 3--sphere admits a
collapse, it is possible to obtain a proof of Waldhausen's classification
of the Heegaard splittings of the sphere. This strategy can be probably
applied to other Seifert manifolds.

\begin{prob}
Give a geometric proof of the classification of Heegaard splittings of
Seifert manifolds.
\end{prob}

\bibliographystyle{gtart}
\bibliography{link}

\end{document}